\documentclass[onecolumn,reqno]{amsart}
\usepackage{amsthm,amsmath,amssymb,amsfonts}
\usepackage{graphicx,color}
\usepackage{appendix}

\newtheorem{theorem}{Theorem}[section]
\newtheorem{lemma}[theorem]{Lemma}
\newtheorem{corollary}[theorem]{Corollary}
\newtheorem{proposition}[theorem]{Proposition}

\theoremstyle{definition}
\newtheorem{definition}[theorem]{Definition}
\newtheorem{example}[theorem]{Example}
\newtheorem{remark}[theorem]{Remark}
\newtheorem{problem}[theorem]{Problem}
\date{\today}

\def \RR {\mathbb R}

\def \beq {\begin {equation}}
\def \eeq {\end{equation}}

\def \W {\widetilde}

\begin{document}	

		\title[Bounds on Discrete Potentials of Spherical $(k,k)$-Designs]{Bounds on Discrete Potentials of Spherical $(k,k)$-Designs}

\author[S. Borodachov]{S. V. Borodachov}
\address{Department of Mathematics, Towson University, 7800 York Rd, Towson, MD, 21252, USA}
\email{sborodachov@towson.edu}

\author[P. Boyvalenkov]{P. G. Boyvalenkov}
\address{ Institute of Mathematics and Informatics, Bulgarian Academy of Sciences,
8 G Bonchev Str., 1113  Sofia, Bulgaria}
\email{peter@math.bas.bg}

\author[P. Dragnev]{P. D. Dragnev}
\address{ Department of Mathematical Sciences, Purdue University \\
Fort Wayne, IN 46805, USA }
\email{dragnevp@pfw.edu}

\author[D. Hardin]{D. P. Hardin}
\address{ Center for Constructive Approximation, Department of Mathematics \\
Vanderbilt University, Nashville, TN 37240, USA }
\email{doug.hardin@vanderbilt.edu}

\author[E. Saff]{E. B. Saff}
\address{ Center for Constructive Approximation, Department of Mathematics \\
Vanderbilt University, Nashville, TN 37240, USA }
\email{edward.b.saff@vanderbilt.edu}

\author[M. Stoyanova]{M. M. Stoyanova} 
\address{ Faculty of Mathematics and Informatics, Sofia University ``St. Kliment Ohridski"\\
5 James Bourchier Blvd., 1164 Sofia, Bulgaria}
\email{stoyanova@fmi.uni-sofia.bg}

\begin{abstract}  
We derive universal lower and upper bounds for max-min and min-max  problems (also known as polarization) for the potential of spherical $(k,k)$-designs
and provide certain examples, including unit-norm tight frames, that attain these bounds. The universality is understood in the sense that the bounds hold for all spherical $(k,k)$-designs and for a large class of potential functions, and the bounds involve certain nodes and weights that are independent of the potential. When the potential function is $h(t)=t^{2k}$, we prove an optimality 
property of the spherical $(k,k)$-designs in the class of all spherical codes of the same cardinality both for max-min and min-max polarization problems.
\end{abstract}

	\maketitle

%{\bf Keywords}: {\it Discrete potentials, sharp spherical configurations, linear programming, Gauss-Gegenbauer  quadrature, universal bounds on polarization of codes}.
%\vskip 3mm
%{\bf MSC 2020}: {\it 05B30, 52C17, 74G65,  94B65; 05E30, 33C45, 52A40}

\section{Introduction}\label{Intro}

Let $\mathbb{S}^{n-1}$ denote the unit sphere in Euclidean space $\mathbb{R}^n$ with the normalized unit surface measure $\sigma_{n}$. We will call a {\em spherical code} a finite nonempty set $C=\{x_1,x_2,\ldots,x_{N}\} \subset   \mathbb{S}^{n-1}$. 
Consider any function $h:[-1,1] \to (-\infty,+\infty]$, finite and continuous on $(-1,1)$ and continuous in the extended sense at $t=1$ and $t=-1$; that is, $\lim\limits_{t\to 1^-}h(t)=h(1)$ and $\lim\limits_{t\to -1^+}h(t)=h(-1)$, where $h(1)$ and $h(-1)$ can assume the value $+\infty$. We define the {\em discrete $h$-potential} associated with $C$ by
\begin{equation}\label{Eq1} U^h(x,C):= \sum_{i=1}^N h(x \cdot x_i), 
\end{equation}
where $x \in \mathbb{S}^{n-1}$ is arbitrary and $x \cdot y$ denotes the Euclidean inner product in $\mathbb{R}^n$. Here we agree that for every $t \in \RR\cup\{+\infty\}$, $t+\infty=\infty$. We allow an infinite value for $h(1)$, since some widely used potentials such as the Coulomb potential, the Riesz potential, and the logarithmic potential assume the value $+\infty$ in \eqref{Eq1} when $x=x_i$ for some $i$. We will also consider even symmetrizations of functions $h$ defining these potentials. The corresponding potentials will assume the value $+\infty$ also when $x=-x_i$ for some $i$. Let 
\begin{equation} \label{max-min}
m^h(C):=\inf \{ U^h(x,C) :\, x \in \mathbb{S}^{n-1}\}, \ \ 
m_{N}^h := \sup  \{ m^h(C) :\, |C|=N, \ C \subset \mathbb{S}^{n-1} \}
\end{equation}
be the {\em max-min polarization quantities} associated with a given code
and cardinality $N$, respectively. Similarly, we define the {\em min-max polarization quantities}
\begin{equation} \label{min-max}
M^h(C):=\sup \{ U^h(x,C) :\, x \in \mathbb{S}^{n-1} \}, \ \ 
M_{N}^h =\inf \{ M^h(C) :\, |C|=N, \ C \subset \mathbb{S}^{n-1} \}.
\end{equation}
To avoid the trivial case of $M^h(C)=+\infty$ for every code $C\neq \emptyset$, we will additionally suppose that $h$ is continuous and finite on $[-1,1]$ when considering quantities \eqref{min-max}. Notable examples here are the Gaussian potential and the $p$-frame potential. 

The problem of finding quantities $m^h_N$ and $M^h_N$ and spherical codes $C$ attaining the supremum in \eqref{max-min} (resp., infimum in \eqref{min-max}) is known as the max-min (resp., min-max) polarization problem. It can be traced back to the problem of finding the Chebyshev polynomial \cite{Che1859}; that is, the monic polynomial of a given degree $N$ with the least uniform norm on $[-1,1]$. Taking the function $-\ln\left|t\right|$ of the polynomial, one obtains the max-min polarization for the logarithmic potential, see also \cite[Chapter VI]{Sze1975}. For this reason, quantity $m^h_N$ is called the Chebyshev constant (of the sphere $\mathbb S^{n-1}$). Early results on polarization were established by Ohtsuka \cite{Oht1967} and Stolarsky \cite{Sto1975circle} in 1960-70s. Later, this problem regained interest of mathematicians starting from the works \cite{Amb2009,AmbBalErd2013,NR1}. A more extensive review can be found in \cite[Chapter 14]{BHS}. The max-min polarization problem can be interpreted as the problem of positioning $N$ sources of light on a sphere so that the darkest spot on the sphere receives the largest possible amount of light. Closely related interpretation of the max-min polarization problem is finding the minimal number of injectors of a substance and their positions in a given area of a body so that every point in that area receives at least a prescribed amount of the substance. Of course, similar applications can be found for the min-max polarization as well. If $N$-point configurations are restricted to a given compact set $A$ and the minimum of their potential is taken over a compact set $D$, where $A\neq D$, we obtain a two-plate polarization problem.

Relevant for coding theory is the optimal covering problem (see, e.g., \cite{BorFPC} for a review). It requires finding an $N$-point spherical code with the smallest covering radius. The covering radius of a given spherical code is the minimal common radius of equal closed balls centered at points of the code whose union covers the sphere or the distance to a point on the sphere ``most remote" from a closest (to it) point of the code. The optimal covering problem arises as the limiting case as $s\to \infty$ of the max-min polarization problem for the Riesz potential $h(x\cdot y)=(2-2x\cdot y)^{-s/2}=\left|x-y\right|^{-s}$, $s>0$, $x,y\in S^{n-1}$. Exact solutions for max-min polarization on the sphere are known for arbitrary $N$ on the unit circle $S^1$ and for $N$ up to $n+1$ on $S^{n-1}$, see \cite{HarKenSaf2013,B2} and references therein. Exact solutions for min-max polarization on the sphere are known when a tight spherical design exists for that dimension $n$ and cardinality $N$ and for $n=4$ and $N=120$, see \cite{B,BDHSS-22a} and references therein. Our goal in this paper is providing efficient lower and upper bounds for max-min and min-max polarization quantities on the sphere for general codes and for $(k,k)$-designs.

\subsection{Definition of $(k,k)$-designs} We shall need the family of Gegenbauer (orthogonal) polynomials
$\{P_\ell^{(n)}(t)\}_{\ell=0}^\infty$. By definition, ${\rm deg}\ \!P_\ell^{(n)}=\ell$, $\ell\geq 0$, and $P_\ell^{(n)}$ are pairwise orthogonal on $[-1,1]$ 
with respect to the measure $\mu_{n}$ given by
\begin{equation}\label{mu_n}
d\mu_{n}(t):=\gamma_{n}(1-t^2)^{(n-3)/2}\, dt,
\end{equation}
where the constant $\gamma_{n}$ is chosen such that $\mu_{n}$ is a probability measure on $[-1,1]$. The normalization $P_\ell^{(n)}(1)=1$
uniquely determines the sequence $\{P_\ell^{(n)}(t)\}_{\ell=0}^\infty$.
The connection between measures $\mu_n$ and $\sigma_n$ is given by the basic form of the Funk-Hecke formula (see, e.g., \cite[Eq. (5.1.9)]{BHS}): for every bounded and measurable function $f:[-1,1]\to\RR$ and any $z\in \mathbb S^{n-1}$, 
\begin{equation}\label{FunkHecke}
\int\limits_{\mathbb S^{n-1}}f(z\cdot x)\ \! d\sigma_n(x)=\gamma_n\int\limits_{-1}^{1}f(t)(1-t^2)^{(n-3)/2}\ \! dt.
\end{equation}

Given a code $C=\{x_1,x_2,\ldots,x_{N}\} \subset \mathbb{S}^{n-1}$ we define the 
{\em moments of $C$} as
\begin{equation}\label{moments} \mathcal{M}_\ell(C):=\sum_{i,j=1}^N P_\ell^{(n)}(x_i \cdot x_j), \ \ \ell \in \mathbb{N}.
\end{equation}
As is well known, the Gegenbauer polynomials are positive definite (see e.g. \cite{Sch42}, \cite[Chapter 5, Definition 5.2.10]{BHS}). This means that the moments are
nonnegative; i.e., $\mathcal{M}_\ell(C) \geq 0$ for every positive integer $\ell$. We remark that the inequality $\mathcal M_2(C)\geq 0$ is equivalent to the inequality 
$$
\sum\limits_{i,j=1}^{N}(x_i\cdot x_j)^2\geq \frac {N^2}{n}
$$ 
which is a special case of the Welch bound \cite{Wel2006} on cross-correlation of signals.

If the equality $\mathcal{M}_\ell(C) = 0$ holds for every $\ell=1,\ldots,\tau$, then $C$ is a spherical $\tau$-design \cite{DGS}. 

In this paper, we utilize the concept of {\em spherical $(k,k)$-designs} (cf. \cite{DS89,Lev98A,KP10,KP11,BZ4,Ban17,HW,W2018}) which will play an essential role in finding lower and upper bounds for the polarization quantities \eqref{max-min} and \eqref{min-max} via linear programming. 

\begin{definition} \label{def-k-k} Let $k$ be a positive integer.
A spherical code $C \subset \mathbb{S}^{n-1}$ is called a {\em spherical $(k,k)$-design} if 
$$\mathcal{M}_\ell(C)=0$$
for each even $\ell=2,\ldots,2k$. 
\end{definition}

It is immediate from this definition and \eqref{moments} that $C$ is a spherical $(k,k)$-design if and only if $-C$ is a spherical $(k,k)$-design. The most straightforward examples of $(k,k)$-designs come from spherical designs \cite{DGS}. Indeed, every spherical $(2k+1-\varepsilon)$-design, $\varepsilon \in \{0,1\}$, is a spherical $(k,k)$-design.

\subsection{Known work on $(k,k)$-designs} The notion of $(k,k)$-designs seems to be first considered 
by Delsarte and Seidel in 1989 \cite{DS89}, who studied (weighted) designs of indices\footnote{A spherical code $C \subset \mathbb{S}^{n-1}$ is said to have an index $\ell$ if $\mathcal{M}_\ell(C)=0$; cf. \cite{DS89,AY,BDK,Ban17}.}
from a set $A+B$, where $A$ and $B$ are finite sets of non-negative integers
(see Section 5 in \cite{DS89} for the specific choices of $A$ and $B$ that lead to $(k,k)$-designs). Moreover, in \cite{DS89} a lower bound on the minimum possible cardinality of $(k,k)$-designs was obtained and 
a strong relation to antipodal sets (designs in the real projective spaces)
was revealed. If one has an antipodal spherical $(2k+1)$-design; that is, a $k$-design in the real projective space (see, e.g. \cite{Hoggar1982,Lev92,W2018} and \cite{BZ4,Ban09,Ban17}), and
selects a point from each pair of antipodal points in it, then one obtains a $(k,k)$-design. Conversely, if a $(k,k)$-design does not possess a pair of antipodal points, then it is a half-design of an antipodal $(2k+1)$-design. 
However, there are $(k,k)$-designs that are not antipodal themselves and cannot be extended to antipodal $(2k+1)$-designs by adding the antipodes of all points that originally do not have an antipode in the original design.

In 2010, Kotelina and Pevnyi \cite{KP10} (see also \cite{KP11,KP16}) considered $(k,k)$-designs 
(called semi-designs in \cite{KP10,KP16} and half-designs in \cite{KP11}) and elaborated on some of their extremal properties. 
Bannai et. al. \cite{BZ4} (see also \cite{Ban09,Ban17}) introduced and 
 investigated ``a half of an antipodal design" (as noted above, it is a $(k,k)$-design). Namely, given an antipodal $(2k+1)$-design of cardinality $2N$, they investigated when it is possible to select one point from each of the $N$ pairs of antipodal points so that the resulting code has its center of mass at the origin (which makes it a $1$-design and hence, a $2$-design). Hughes-Waldron \cite{HW} (see also \cite{W2018}) considered spherical ``half-designs" as sequences of points (not necessarily distinct) with even moments $\mathcal{M}_{2i}(C)$, $i=1,2,\dots,k$, equal to zero, which agrees with Definition \ref{def-k-k}. Recently, Elzenaar and Waldron \cite{ElzWal2025} presented numerical constructions of a number of new $(k,k)$-designs that have minimal cardinality.

\begin{remark}
If we allow repetition of points in $C$ 
(this can be viewed as a weighted spherical code), then the 
union $C \cup (-C)$ 
is an antipodal weighted spherical $(2k+1)$-design if and only if $C$ is a weighted $(k,k)$-design (see \cite{DS89,W2018}). 
\end{remark}

Another significant source of examples and probably the most important interpretation of $(k,k)$-designs comes from
the theory of tight frames (see \cite{BF03,W2018}). 
The unit-norm tight frames (considered as sets; i.e., repetition of points is not allowed) are, in fact, spherical $(1,1)$-designs (see, for example,  \cite{W2018} and references therein). Since each $(k,k)$-design
is an $(\ell,\ell)$-design for every positive integer $\ell \leq k$, the 
investigation of $(k,k)$-designs provides a point of view on unit-norm tight frames as nested classes of $(\ell,\ell)$-designs for $\ell \geq 1$.

Our goal in the present paper is to obtain estimates for the following  max-min and min-max polarization quantities
for $(k,k)$-designs:
\begin{equation}\label{PolMinDes}
m^h_{N}(k):=\sup \{ m^h(C)\, : \, |C|=N, \,C \ {\rm is \ a }\ (k,k){\mbox -}{\rm design \ on \ } \mathbb{S}^{n-1}\}
\end{equation}
and
\begin{equation}\label{PolMaxDes}
M^h_{N}(k):=\inf \{ M^h(C)\, : \, |C|=N, \,C \ {\rm is \ a }\ (k,k){\mbox -}{\rm design \ on \ } \mathbb{S}^{n-1} \}.
\end{equation}
For this purpose it is natural to utilize linear programming techniques, since in particular, for any $(k,k)$-design $C \subset \mathbb{S}^{n-1}$ and any even polynomial $f$ of degree at most $2k$, the discrete potential $U^f(x,C)$ is constant on $\mathbb{S}^{n-1}$ (see Lemma \ref{k-k-lemma}). Note that from the definitions \eqref{max-min}-\eqref{min-max} and \eqref{PolMinDes}-\eqref{PolMaxDes} we conclude that whenever a $(k,k)$-design on $\mathbb{S}^{n-1}$ of cardinality $N$ exists,  $m_N^h\geq m_N^h(k)$ and $M_N^h\leq M_N^h(k)$, respectively. A more extensive review of research related to max-min and min-max polarization is given in Section \ref{main}.

\subsection{Properties of $(k,k)$-designs}\label{Pkkd} For any positive integer $\ell$ and $z\in \mathbb{S}^{n-1}$, we denote (see, e.g., Lemma 2.1 in \cite{NS88})
\begin{equation} \label{cell}
c_\ell:= \int_{\mathbb{S}^{n-1}} (z\cdot y)^\ell \, d \sigma_{n} (y)=
\left\{
\begin{array}{ll}
0& \ell\geq 1\ \ \mbox{is odd,} \\
  \frac{1\cdot 3\cdot \dots \cdot(\ell-1)}{n\cdot (n+2)\cdot \dots \cdot (n+\ell-2)}, & \ell\geq 2\ \ \mbox{is even}.
\end{array}
  \right.
\end{equation}
We also set $c_0:=\int_{\mathbb S^{n-1}}1 \ \! d\sigma_n(y)=1$. Note that $c_\ell=\int_{-1}^1 t^\ell d \mu_n(t)$, $\ell\geq 0$.

A spherical code $C=\{x_1,x_2,\ldots,x_{N}\} \subset \mathbb{S}^{n-1}$ 
is said to be a {\em spherical design of order $\ell$}, $\ell\in \mathbb{N}$, if 
the following {\em Waring-type identity} holds for any $x\in \mathbb{R}^n$ (cf. \cite[Eq. (6.8)]{W2018})
\begin{equation}\label{Waring}
Q_\ell (x):=\sum_{i=1}^N (x\cdot x_i)^{\ell}-c_{\ell} \|x\|^{\ell} N
\equiv 0.
\end{equation}
Applying the Laplace operator $\triangle=\partial^2/\partial t_1^2 +\dots+\partial^2/\partial t_{n}^2 $ one obtains  
$$ \triangle Q_\ell(x)=\ell(\ell-1)Q_{\ell-2}(x), $$
which implies that spherical designs of order $\ell$ are also spherical 
designs of order $(\ell-2j)$ for $j=1,\dots,\lfloor (\ell-1)/2 \rfloor$. This observation leads to
an equivalent definition for spherical $(k,k)$-designs. Indeed, applying \eqref{Waring} to $x \in \mathbb{S}^{n-1}$, we see that 
the identity
\begin{equation}\label{Waring-on-sphere}
\sum_{i=1}^N (x\cdot x_i)^{\ell}=c_{\ell}N 
\end{equation}
holds for any spherical design of order $\ell$. In particular, if $C \subset \mathbb{S}^{n-1}$ is a spherical $(k,k)$-design, it holds for $\ell=2,4,\ldots,2k$. 
Observe also that if \eqref{Waring} holds for two consecutive positive integers, say $\tau$ and $\tau-1$, then the code $C$ is a spherical $\tau$-design. This discussion is closely related to \cite[Section 3]{GS81}.

Next we summarize some equivalent definitions for spherical $(k,k)$-designs, which we will use later in the paper.
% We already saw that the case of equality is very important as $(k,k)$-designs can be defined via the equalities $\mathcal{M}_{2i} \geq 0$
%for $i=1,\ldots,k$. 

Recall that a real-valued function on $\mathbb{S}^{n-1}$ is called a {\em spherical harmonic of degree $\ell$} if it is the restriction of a homogeneous polynomial $Y$ in $n$ variables of degree $\ell$ that is harmonic, i.e. for which $\triangle Y \equiv 0$. Denote by $\Pi_m$ the space of real univariate polynomials of degree at most $m$, by $\mathbb{P}_\ell^n$ the space of homogeneous polynomials of degree $\ell$ in $n$ variables, and by $\mathbb{H}_\ell^n$ the subspace consisting of spherical harmonics of degree $\ell$.  

\begin{lemma}\label{k-k-lemma} \rm{(see, e.g. \cite[Chapter 6]{W2018}, \cite[Lemma 5.2.2]{BHS}, \cite{BOT15})} 
Let $n\geq 2$, $C=\{x_1,\dots,x_{N}\}\subset \mathbb{S}^{n-1}$ be a 
spherical code, and $k \in \mathbb{N}$. Then the following are equivalent.
\begin{itemize}
\item[i)] $C$ is a spherical $(k,k)$-design;
\vskip 1mm
\item[ii)] $C$ is a spherical design of order $2k$;
\vskip 1mm
\item[iii)] the identity \eqref{Waring} holds for $\ell=2k$ and any $x\in \mathbb{S}^{n-1}$;
\vskip 1mm
\item[iv)] the moments $\mathcal{M}_{2i}(C)=0$ for $i \in \{1,\ldots,k\}$;
\vskip 1mm
\item[v)] the identity 
\begin{equation} \label{k-k-constant}
U^f(x,C)=\sum_{y \in C} f(x \cdot y) = f_0|C|=|C| \int_{-1}^1 f(t) d\mu_n(t)
\end{equation} 
holds for every (even) polynomial $f(t)=\sum_{i=0}^{k} f_{2i}P_{2i}^{(n)}(t) \in \Pi_{2k}$;
\vskip 1mm
\item[vi)] for any $Y \in \cup_{i=1}^k \mathbb{H}_{2i}^n$,
$$ \sum_{j=1}^N Y(x_j)=0; $$
\item[vii)] $-C$ is a spherical $(k,k)$-design. 
\end{itemize}
\end{lemma}

\subsection{Organization of the paper} In Section \ref{Genbounds}, formulations of linear programs to be solved later are given and general linear programming lower and upper bounds are developed. 
Section \ref{Prelim} contains necessary background material: degree of precision of Gauss-Gegenbauer quadratures and properties of orthogonal polynomials with respect to certain signed measures that are necessary for solving the linear programs. Section~\ref{main} contains the main results of this paper, bounds \eqref{PolarizationULB-e1}, \eqref{PolarizationULB-e2}, \eqref{PolarizationUUB-e}, and \eqref{PolarizationUUB-e3} and Theorem \ref{minmax-maxmin-extract}. The proofs of the universal bounds in \eqref{PolarizationULB-e1}, \eqref{PolarizationULB-e2}, \eqref{PolarizationUUB-e}, and \eqref{PolarizationUUB-e3} are found in Sections~\ref{Ubounds} and~\ref{section-UUB} (lower and upper
bounds, respectively), where the conditions for their attaining are also provided. In Section \ref{kk-optimality}, we consider the special case $h(t)=t^{2k}$ and
prove that spherical $(k,k)$-designs (when they exist) provide 
the best $h$-polarization, solving the problem 
of finding quantities $m^h_N$ and $M^h_N$ in this case (see Theorem \ref{minmax-maxmin-extract}).

\section{General linear programming bounds for polarization of spherical $(k,k)$-designs}\label{Genbounds}

\subsection{Constant polynomial potentials}

Similarly to the case of spherical $\tau$-designs, $(k,k)$-designs have a discrete potential which is constant over the sphere for polynomial potential functions of a certain type. This follows directly from \eqref{Waring-on-sphere} and is given by identity \eqref{k-k-constant}
from Lemma \ref{k-k-lemma} 
which holds for every $(k,k)$-design $C \subset \mathbb{S}^{n-1}$ and 
every even polynomial $f \in \Pi_{2k}$ with Gegenbauer expansion $f(t)=\sum_{i=0}^{k} f_{2i}P_{2i}^{(n)}(t)$.

Property \eqref{k-k-constant} allows the derivation of Delsarte-Yudin type (cf. \cite{Del1,Del2,Del3,DGS,Y}) lower and upper polarization bounds as shown in the next two subsections. 

\subsection{Lower bounds}

We first define a class of polynomials which is suitable for linear 
programming lower bounds of Delsarte-Yudin type. 

\begin{definition} \label{L-k-k}
Given a potential function $h$ as in \eqref{Eq1} and integers $n \geq 2$ and $k\geq 1$, denote by $\mathcal{L}(k;h)$ the set of polynomials
\begin{equation} \label{L(k,h)}
\mathcal{L}(k;h):=\left\{f(t)=\sum_{i=0}^{k} f_{2i}P_{2i}^{(n)}(t) \in \Pi_{2k}: f(t) \leq h(t), t \in [-1,1] \right\}. \end{equation}
\end{definition}

Note that the set $\mathcal{L}(k;h)$ is nonempty for every $k$ and $h$. 

The (folklore) general linear programming lower bound is immediate. We present a proof for completeness. 

\begin{theorem}\label{PULB-k-k-general}
Given $h$ as in \eqref{Eq1}, $n \geq 2$, and $k\geq 1$, let $C \subset \mathbb{S}^{n-1}$ be a spherical $(k,k)$-design,
$|C|=N$, and let $f \in  \mathcal{L}(k;h)$. 
Then 
\begin{equation} \label{k-k-constantL}
U^h(x,C) \geq f_0N,\ \ \ x\in \mathbb S^{n-1},
\end{equation} 
and, consequently, 
\begin{equation} \label{pulb-1}
m_{N}^h \geq m_{N}^h(k) \geq N \cdot \sup \{ f_0 : f \in \mathcal{L}(k;h) \}.
\end{equation}
\end{theorem}

\begin{proof} 
Let $C=\{x_1,\ldots,x_N\} \subset \mathbb{S}^{n-1}$ be a spherical $(k,k)$-design, $f \in  \mathcal{L}(k;h)$, and $x \in \mathbb{S}^{n-1}$. Then we consecutively have 
\begin{eqnarray*}
    U^h(x,C) &=& \sum_{i=1}^N h(x \cdot x_i) \\
           &\geq& \sum_{i=1}^N f(x \cdot x_i) \\
           &=& f_0N.
\end{eqnarray*}
We used the fact that $h(t) \geq f(t)$ for every $t \in [-1,1]$ for the inequality and identity 
\eqref {k-k-constant} for the last step.  

Since this is true for every point $x\in \mathbb{S}^{n-1}$, every polynomial $f \in \mathcal{L}(k;h)$ and every $(k,k)$-design $C$ with $|C|=N$, we conclude that
$$ m_{N}^h(k) \geq N \cdot \sup \{ f_0 : f \in \mathcal{L}(k;h) \}. $$
Finally, we have 
$ m_{N}^h \geq m_{N}^h(k) $
from defining equalities \eqref{max-min} and \eqref{PolMinDes}. 
\end{proof}

Thus, every polynomial from the set $\mathcal{L}(k;h)$ provides a lower bound on 
the general potential $U^h(\ \!\cdot\ \!,C)$, This bound is valid for every $(k,k)$-design $C \subset \mathbb{S}^{n-1}$ and gives rise to the following linear program:
\begin{equation}\label{PULB_LP_Program} 
\left\{ \begin{array}{rr}
{\rm given}: & \quad n, \ k,  \ h  \\
{\rm maximize}: & \quad f_0  \\
{\rm subject \ to:} &\quad  f \in \mathcal{L}(k;h).
 \end{array}  \right.
\end{equation}

In Section \ref{Ubounds}, we solve program \eqref{PULB_LP_Program}
for potentials given by functions of inner product squared whose $(k+1)$-th derivatives do not change sign in $(0,1)$. In Section \ref{kk-optimality}, we consider the special case $h(t)=t^{2k}$ and show that $(k,k)$-designs are optimal with respect to bounds \eqref{k-k-constantL} and  \eqref{pulb-1} in this case.

\subsection{Upper bounds} \label{UUbounds}

We now derive upper bounds of Delsarte-Yudin type for polarization of $(k,k)$-designs. We start again by defining an appropriate class of polynomials. 

\begin{definition} \label{U-k-k}
Given a potential function $h$ as in \eqref{Eq1}, integers $n \geq 2$ and $k\geq 1$, and a number $s\in [0,1]$, denote by $\mathcal{U}(k,s;h)$ the set of polynomials
$$ \mathcal{U}(k,s;h):=\left\{q(t)=\sum_{i=0}^{k} q_{2i}P_{2i}^{(n)}(t) 
\in \Pi_{2k}:\, q(t) \geq h(t), \ t \in [-s,s] \right\}. $$
\end{definition}
Note that the set $\mathcal{U}(k,s;h)$ is non-empty if $s<1$ or if $s=1$ and $h(1),h(-1)<\infty$.

Given a code $C\subset \mathbb{S}^{n-1}$ and a point 
$x\in \mathbb{S}^{n-1}$, we denote by 
$$ T(x,C):= \{ x \cdot y : y \in C \} $$
the collection of inner products of $x$ with the points of $C$ and define 
$$r(C):= \min_{x\in \mathbb{S}^{n-1}} \max \ \! T(x,C\cup (-C)).$$

\begin{remark}
We remark that $r(C)$ is the covering radius (represented in terms of inner products) of $C\cup(-C)$. 
\end{remark}

Assume that $N$ is such that there exists a spherical $(k,k)$-design 
on $\mathbb{S}^{n-1}$. Then we define
\begin{equation}\label{R_{k,N}}
\begin{split}
    r_{k,N} &:= \inf \{ r(C) \, :\, C \ {\rm is \ a}\ (k,k)\mbox{-design} \text{ and } |C|=N\}, \\
R_{k,N} &:= \sup \{ r(C) \, :\, C  \ {\rm is \ a}\ (k,k)\mbox{-design}\text{ and } |C|=N\}.
\end{split}
\end{equation}

The counterpart of Theorem \ref{PULB-k-k-general} for upper bounds is below. 

\begin{theorem}\label{PUUB-k-k-general}
Given $h$ as in \eqref{Eq1}, let $C \subset \mathbb{S}^{n-1}$ be a spherical $(k,k)$-design, 
$|C|=N$, and $n \geq 2$, $k\geq 1$, and $s\in (0,1]$ be such that $\mathcal{U}(k,s;h)\neq \emptyset$. Then the following holds:

\begin{itemize}
\item[i)] If $s>r(C)$, then for every $q\in \mathcal{U}(k,s;h)$,
\begin{equation} \label{k-k-constantU-2}
m^h(C)=\inf_{x \in \mathbb{S}^{n-1}} U^h(x,C) \leq q_0N.
\end{equation} 

\item[ii)] If $s>R_{k,N}$, then 
%\begin{equation} \label{puub-2}
$$ m_{N}^h(k) \leq  N \cdot \inf \{ q_0 : q \in \mathcal{U}(k,s;h) \}. $$
%\end{equation}

\item[iii)] When $s=1$ and both $h(1)$ and $h(-1)$ are finite, for every $q\in \mathcal U(k,1;h)$, we have 
%\begin{equation} \label{k-k-constantU-3}
$$ M^h(C)=\sup_{x \in \mathbb{S}^{n-1}} U^h(x,C) \leq q_0N. $$
%\end{equation} 
\end{itemize}
\end{theorem}

\begin{remark}
Since the set $\mathcal U(k,s;h)$ can only expand as $s$ decreases, the value of the infimum in ii) decreases as $s$ approaches $R_{k,N}$ from the right. Then Theorem \ref{PUUB-k-k-general} ii) implies that
$$
m_N^h(k)\leq N\ \cdot\!\!\!\!\!\lim\limits_{s\to (R_{k,N})^+}\!\!\inf\{q_0 : q\in \mathcal U(k,s;h)\}.
$$
\end{remark}

{\bf Proof of Theorem \ref{PUUB-k-k-general}.}
Assume first that $s>r(C)$. Then there exists a point $y\in \mathbb S^{n-1}$ such that $\max T(y,C\cup (-C))<s$. Consequently, for every $i=1,\ldots,N$, we have $-s<x_i\cdot y<s$. Since $q(t)\geq h(t)$, $t\in [-s,s]$, and $C$ is a $(k,k)$-design, by Lemma \ref{k-k-lemma} v) we have
$$
m^h(C)=\inf\limits_{x\in \mathbb S^{n-1}}\sum\limits_{i=1}^{N}h(x\cdot x_i)\leq \sum\limits_{i=1}^{N}h(y\cdot x_i)\leq \sum\limits_{i=1}^{N}q(y\cdot x_i)=q_0N,
$$
which proves i). Assume next that $s>R_{k,N}$. Then $s>r(C)$ for every $N$-point $(k,k)$-design $C\subset \mathbb S^{n-1}$. From i), for each $q\in \mathcal U(k,s;h)$, we have
$$
m^h_N(k)=\sup\limits_{C - (k,k)\text{-design}\atop |C|=N}m^h(C)\leq q_0N.
$$
Taking now the infimum over $q\in \mathcal U(k,s;h)$, we have
$m^h_N(k)\leq N\cdot \inf\{q_0 : q\in \mathcal U(k,s;h)\}$.

Finally, assume that $s=1$ and $h(1),h(-1)<\infty$. Since $C$ is a $(k,k)$-design, by Lemma \ref{k-k-lemma} v) we obtain that
$$
M^h(C)=\sup\limits_{x\in \mathbb S^{n-1}}\sum\limits_{i=1}^{N}h(x\cdot x_i)\leq \sup\limits_{x\in \mathbb S^{n-1}}\sum\limits_{i=1}^{N}q(x\cdot x_i)=q_0N,
$$
which proves iii).\hfill$\Box $

This gives rise to the following linear program:
\begin{equation}\label{PUUB_LP_Program} 
\left\{ 
 \begin{array}{rr}
{\rm given}: & \quad n, \ k, \ s, \ h  \\
{\rm minimize}: & \quad q_0  \\
{\rm subject \ to:} &\quad  q \in \mathcal{U}(k,s;h).
 \end{array}  \right.
\end{equation}

In Section \ref{section-UUB}, we solve program \eqref{PUUB_LP_Program} for potentials given by functions of inner product squared whose $(k+1)$-th derivatives do not change sign in $(0,1)$. The solution depends on the existence of a $(k,k)$-design and a proper  
choice of the range for $s$. In Section \ref{kk-optimality}, we consider the special case $h(t)=t^{2k}$, and show that $(k,k)$-designs  are optimal with
respect to bound \eqref{k-k-constantU-2} when they exist.

\section{Some background material}\label{Prelim}

This section contains the background material used to state and prove the main results of this paper. 

\subsection{Gauss-Gegenbauer quadratures}\label{sGJQ} A quadrature (or a quadrature formula) is any weighted sum of the form $\sum\limits_{j=1}^{M}b_jf(t_j)$ used to compute approximately a given definite integral of the function $f$. A quadrature $\sum\limits_{j=1}^{M}b_jf(t_j)$ is called exact on a function $f$ if the sum $\sum\limits_{j=1}^{M}b_jf(t_j)$ equals the value of the given definite integral for $f$. We will recall the following classical quadrature and discuss its exactness. 

It is known \cite[Theorem 3.3.1]{Sze1975} that the $k+1$ zeros of the Gegenbauer polynomial $P_{k+1}^{(n)}$ are all real, distinct, and lie in $(-1,1)$ (see also the case $s=1$ of Lemma \ref{int_s}). 
Throughout the paper, denote by $-1<\alpha_1<\ldots<\alpha_{k+1}<1$
the zeros of $P_{k+1}^{(n)}$. Since the density defining the measure $\mu_n$ in \eqref{mu_n} is an even function, the polynomial $P_{k+1}^{(n)}$ is even if $k+1$ is even and odd if $k+1$ is odd. This fact implies that $\alpha_i$'s are symmetric about $0$. Let $L_i$, $i=1,\ldots,k+1$, denote the fundamental Lagrange polynomials for the system of nodes $\{\alpha_1,\ldots,\alpha_{k+1}\}$; that is, $L_1,\ldots,L_{k+1}$ have degree $k$ and $L_i(\alpha_j)=0$ if $j\neq i$
and $L_i(\alpha_j)=1$ if $i=j$. They can be computed as \cite[Equation (14.1.2)]{Sze1975}
$$
L_i(t)=\frac {P_{k+1}^{(n)}(t)}{(t-\alpha_i)(P_{k+1}^{(n)})'(\alpha_i)},\ \ \ i=1,\ldots,k+1.
$$

The Gauss-Gegenbauer quadrature is the following quadrature for computing approximately the integral on the left-hand side: 
\begin{equation} \label{GJQ}
\gamma_{n} \int_{-1}^1 f(t) (1-t^2)^{(n-3)/2} \, dt\approx\sum_{j=1}^{k+1} \rho_j f(\alpha_j),
\end{equation}
where the weights $\rho_i$ are defined as
\begin{equation}\label{rho_weights'}
\rho_i:=\gamma_n\int\limits_{-1}^{1}L_i(t)(1-t^2)^{(n-3)/2}\ \! dt,\ \ \ i=1,\ldots,k+1.
\end{equation}
Since $\rho_i=\rho_iL_i(\alpha_i)=\sum_{j=1}^{k+1}\rho_jL_i(\alpha_j)$, equality holds in \eqref{GJQ} for $f=L_i$, $i=1,\ldots,k+1$. The polynomials $L_1,\ldots,L_{k+1}$ form a basis for the space $\Pi_k$. Therefore, by linearity, we have ``=" in \eqref{GJQ} for every polynomial $f\in \Pi_k$; that is quadrature \eqref{GJQ} is exact for every polynomial of degree up to $k$. Quadratures with weights defined as in \eqref{rho_weights'} for an arbitrary system of $k+1$ nodes are called interpolatory quadratures. Quadrature \eqref{GJQ} whose nodes are zeros of $P_{k+1}^{(n)}$ is, in fact, exact on polynomials of a higher degree.

\begin{lemma}\label{LemmaGJQ}(\cite[Theorems 3.4.1 and 3.4.2]{Sze1975})
For every $n\geq 2$ and $k\geq 1$, quadrature \eqref{GJQ} is exact for every polynomial from $\Pi_{2k+1}$. 
Furthermore, $\rho_i>0$ and $\rho_i=\rho_{k+2-i}$, $i=1,\ldots,k+1$.
\end{lemma}
\begin{proof} If $f\in \Pi_{2k+1}$ is any polynomial, there exist polynomials $q,r\in \Pi_k$ such that $$
f(t)=P_{k+1}^{(n)}(t)q(t)+r(t).
$$
Since $P_{k+1}^{(n)}$ is orthogonal to $\Pi_k$ with respect to the measure $\mu_n$ (see notation \eqref{mu_n}), we have
$$
\gamma_n\int\limits_{-1}^{1}f(t)(1-t^2)^{(n-3)/2}\ \! dt=\int\limits_{-1}^{1}P_{k+1}^{(n)}(t)q(t) \! d\mu_n(t)+\int\limits_{-1}^{1}r(t)\ \! d\mu_n(t)=\int\limits_{-1}^{1}r(t)\ \! d\mu_n(t).
$$
On the other hand, since $\alpha_1,\ldots,\alpha_{k+1}$ are zeros of $P_{k+1}^{(n)}$, we have
$$
\sum\limits_{j=1}^{k+1}\rho_jf(\alpha_j)=\sum\limits_{j=1}^{k+1}\rho_jP_{k+1}^{(n)}(\alpha_j)q(\alpha_j)+\sum\limits_{j=1}^{k+1}\rho_jr(\alpha_j)=\sum\limits_{j=1}^{k+1}\rho_jr(\alpha_j).
$$
Since quadrature \eqref{GJQ} was shown to be exact on $\Pi_k$ and $r\in \Pi_k$, we have
$$
\int\limits_{-1}^{1}r(t)\ \! d\mu_n(t)=\sum\limits_{j=1}^{k+1}\rho_jr(\alpha_j).
$$
Then quadrature \eqref{GJQ} is exact on $f$. In view of arbitrariness of $f\in \Pi_{2k+1}$, quadrature \eqref{GJQ} is exact on $\Pi_{2k+1}$.

Since quadrature \eqref{GJQ} remains exact on the squares of the fundamental Lagrange polynomials $L_i(t)$, $i=1,\ldots,k+1$, for the system of nodes $\{\alpha_1,...,\alpha_{k+1}\}$, we have 
$$
\rho_i=\sum\limits_{j=1}^{k+1}\rho_j(L_i(\alpha_j))^2=\int\limits_{-1}^{1}(L_i(t))^2d\mu_n(t)>0, \ \ i=1,...,k+1.
$$
Since $P_{k+1}^{(n)}$ has the same parity as $k+1$, we have 
\begin{equation*}
\begin{split}
L_i(-t)&=\frac{P_{k+1}^{(n)}(-t)}{(-t-\alpha_i)\left(P_{k+1}^{(n)}\right)^\prime(\alpha_i)}=\frac{P_{k+1}^{(n)}(t)}{(t+\alpha_i)\left(P_{k+1}^{(n)}\right)^\prime(-\alpha_i)}\\
&=\frac{P_{k+1}^{(n)}(t)}{(t-\alpha_{k+2-i})\left(P_{k+1}^{(n)}\right)^\prime(\alpha_{k+2-i})}=L_{k+2-i}(t),
\end{split}
\end{equation*}
from which using definition \eqref{rho_weights'} we can easily obtain that $\rho_i=\rho_{k+2-i}$.
\end{proof}

Consider now the set $-1=\beta_0<\beta_1<\ldots<\beta_k<\beta_{k+1}=1$, where $\beta_1,\ldots,\beta_k$ are zeros of the Gegenbauer polynomial $P_k^{(n+2)}$. They are all real, simple, lie in $(-1,1)$ (see \cite[Theorem 3.3.1]{Sze1975}) and are symmetric about $0$. Denote by $\W L_i$, $i=0,\ldots,k+1$, the fundamental Lagrange polynomials for the system of nodes $\{\beta_0,\ldots,\beta_{k+1}\}$ and let
\begin{equation}\label{delta_i}
\delta_i:=\gamma_n\int\limits_{-1}^{1}\W L_i(t)(1-t^2)^{(n-3)/2}\ \! dt,\ \ \ i=0,\ldots,k+1.
\end{equation}
Using the explicit formulas for $\W L_i$'s in terms of $P_k^{(n+2)}$, we have
\begin{eqnarray} %\label{delti}
\delta_i &:=& \gamma_n\int_{-1}^1 \frac{(t^2-1)P_k^{(n+2)}(t)}{(t-\beta_i)(\beta_i^2-1)\left(P_k^{(n+2)}\right)^\prime (\beta_i)}(1-t^2)^{(n-3)/2}\, dt,\quad i=1,\dots,k, \label{delti-1} \\
\delta_0 &:=& \gamma_n\int_{-1}^1 \frac{(t-1)P_k^{(n+2)}(t)}{-2P_k^{(n+2)}(-1)}(1-t^2)^{(n-3)/2}\, dt, \\ %\label{delti-2} \\ 
\delta_{k+1} &:=& \gamma_n\int_{-1}^1 \frac{(t+1)P_k^{(n+2)}(t)}{2P_k^{(n+2)} (1)}(1-t^2)^{(n-3)/2}\, dt. \label{delti-3}
\end{eqnarray}
Analogously to Lemma \ref{LemmaGJQ}, one can prove the following.
\begin{lemma}\label{LemmaGJQ1}
Given $n\geq 2$ and $k\geq 1$, the quadrature formula
\begin{equation}\label{GJQ_beta'}
\gamma_{n} \int_{-1}^1 f(t) (1-t^2)^{(n-3)/2} \, dt\approx \sum_{i=0}^{k+1} \delta_i f(\beta_i)
\end{equation}
is exact on $\Pi_{2k+1}$. Furthermore, $\delta_i>0$ and $\delta_{k+1-i}=\delta_i$, $i=0,\ldots,k+1$.
\end{lemma}
\begin{proof}
The proof is similar to that of Lemma \ref{LemmaGJQ}. Take any polynomial $f\in\Pi_{2k+1}$. There exist polynomials $q\in \Pi_{k-1}$ and $r\in \Pi_{k+1}$ such that
$
f(t)=(1-t^2)P_k^{(n+2)}(t)q(t)+r(t).
$
Then
\begin{equation*}
\begin{split}
\gamma_{n} \int_{-1}^1\!\! f(t) (1-t^2)^{(n-3)/2} \, dt&=\gamma_{n} \int_{-1}^1 \!\!\! P_k^{(n+2)}(t)q(t) (1-t^2)^{(n-1)/2} \, dt+\gamma_{n} \int_{-1}^1\!\!\! r(t) (1-t^2)^{(n-3)/2} \, dt\\
=&\gamma_{n} \int_{-1}^1\!\!\! r(t) (1-t^2)^{(n-3)/2} \, dt.
\end{split}
\end{equation*}
On the other hand, since the polynomial $(1-t^2)P_k^{(n+2)}(t)$ vanishes at each $\beta_i$, $i=0,\ldots,k+1$, we have
$$
\sum_{i=0}^{k+1}\delta_if(\beta_i)=\sum_{i=0}^{k+1}\delta_ir(\beta_i).
$$
Since quadrature \eqref{GJQ_beta'} is interpolatory, it is exact on $r\in\Pi_{k+1}$. Then it is exact on any $f\in \Pi_{2k+1}$.
To prove that $\delta_{k+1-i}=\delta_i$, $i=0,\ldots,k+1$, we observe that $\W L_i(-t)=\W L_{k+1-i}(t)$ due to the symmetry of the nodes $\beta_i$. To show the positivity of the weights, we let $v_i(t):=\frac {(1-\beta_i^2)(\W L_i(t))^2}{1-t^2}$, $i=1,\ldots,k$, which has a degree $2k$. We also let $v_0(t):=\frac{2(\W L_0(t))^2}{1-t}$ and $v_{k+1}(t):=\frac {2(\W L_{k+1}(t))^2}{1+t}$, both of which have degree $2k+1$. Then
$$
\delta_i=\delta_iv_i(\beta_i)=\sum_{j=0}^{k+1}\delta_jv_i(\beta_j)=
\int_{-1}^{1}v_i(t)\ \! d\mu_n(t)>0,\ \ \ i=0,\ldots,k+1,
$$
completing the proof.
\end{proof}

\subsection{Positive-definite signed measures and the corresponding quadratures}\label{PDmeasures}

We recall the following definition (see \cite[Definition 3.4]{CK} and \cite[Definition 2.1]{BDHSS-DCC19}).

\begin{definition}\label{nu_pos}
A signed Borel measure $\nu$ on $\mathbb{R}$ for which all polynomials are integrable is called {\em positive definite up to degree $m$}  if for all real polynomials $p\not\equiv 0$ of degree at most $m$, we have $ \int p(t)^2\, d\nu(t) >0$.
\end{definition} 

The following lemma investigates the positive definiteness of the signed measure 
\begin{equation}\label{nu_s'}
d\nu_s (t):= \gamma_{n} (s^2-t^2)(1-t^2)^{(n-3)/2}\, dt
\end{equation}
in terms of $s$.

\begin{lemma} \label{SignedNu_s}  Let $k\geq 1$ and let $\alpha_{k+1}< s\leq 1$, where $\alpha_{k+1}$ is the largest root of the $(k+1)$-th Gegenbauer polynomial $P_{k+1}^{(n)}(t)$. Then the signed measure $\nu_s$ is positive definite up to degree $k-1$.
\end{lemma}

\begin{proof} Let $ q(t)\in \Pi_{k-1}$, $q(t) \not\equiv 0$. Using the fact that Gauss-Gegenbauer quadrature \eqref{GJQ} is exact on $\Pi_{2k+1}$ (see Lemma \ref{LemmaGJQ}) and the fact that the polynomial $(q(t))^2(s^2-t^2)$ has degree at most $2k$, we obtain
$$ \gamma_{n} \int_{-1}^1 (q(t))^2 (s^2-t^2)(1-t^2)^{(n-3)/2}\, dt=\sum_{i=1}^{k+1} \rho_i (q(\alpha_i))^2(s^2-\alpha_i^2)>0,$$
where the positivity follows because $s> \alpha_{k+1}>\ldots>\alpha_1>-s$ (recall that the zeros $\{\alpha_i\}$ of $P_{k+1}^{(n)}$ are symmetric about the origin), $\rho_i>0$, $i=1,\ldots,k+1$ (see Lemma \ref{LemmaGJQ}), and $q$ cannot vanish for all $\alpha_i$ because its degree is at most $k-1$. This proves the lemma.
\end{proof}

\begin{remark}
    Note that if $s=\alpha_{k+1}$, then the conclusion in the display above fails for the polynomial $q(t)=P_{k+1}^{(n)}(t)/[(t-\alpha_1)(t-\alpha_{k+1})]$ of degree $k-1$.
\end{remark}

Assume again that $\alpha_{k+1}<s\leq 1$. From Definition \ref{nu_pos} and Lemma \ref{SignedNu_s} we have that 
$$ \langle f,g \rangle_{\nu_s} := \int_{-1}^1 f(t) g(t) \, d \nu_s (t) $$
defines an inner product on $\Pi_{k-1}$. Therefore, 
given a basis $\{p_j\}_{j=0}^{k}$ of $\Pi_{k}$, the Gram-Schmidt orthogonalization procedure \cite[Section 6.4]{LayLA}
\begin{equation} \label{sOrtho}
q_{0,s}(t)=p_0(t),\quad q_{j,s}(t)=p_j(t)-\sum_{i=0}^{j-1}\frac{\langle p_j,q_{i,s}\rangle_{\nu_s}}{\langle q_{i,s},q_{i,s} \rangle_{\nu_s}} \cdot q_{i,s}(t),\quad j=1,\dots,k,
%\nonumber
\end{equation}
yields an orthogonal basis $\{q_{0,s},\ldots,q_{k,s}\}$ with respect to $\nu_s$. Note that since $\nu_s$ is an even measure, the polynomials $q_{j,s}$ will be even for even $j$ and odd for odd $j$, an easy consequence of \eqref{sOrtho} if we select $p_j(t)=t^j$, $j=0,\ldots,k$. 

\begin{lemma}\label{int_s}
Let $k\geq 1$ and $\alpha_{k+1}<s\leq 1$. Then the polynomial $q_{k,s}$ has $k$ real simple zeros $-s<\lambda_1<\ldots<\lambda_k<s$ that are symmetric about $0$.
\end{lemma}
\begin{proof}
First, assume to the contrary that $q_{k,s}$ has less than $k$ distinct real zeros. Let $B$ be the set of zeros of $q_{k,s}$ having an odd multiplicity.
Denote $v(t):=\prod\limits_{z\in B}(t-z)$. If $B=\emptyset$, we let $v(t):=1$. Then $v$ has degree at most $k-1$. Furthermore, the polynomial $q_{k,s}(t)v(t)$ can only have real zeros of even multiplicities and may contain irreducible quadratic factors. Then its degree is at most $2k-2$. Since every monic irreducible quadratic factor is the sum of a square of a degree one polynomial and a positive constant, the polynomial $q_{k,s}(t)v(t)$ equals a non-zero constant times the sum of squares of polynomials of degrees at most $k-1$. By Lemma \ref{SignedNu_s} and the orthogonality property of $q_{k,s}$, we have
$$
0=\int\limits_{-1}^{1}q_{k,s}(t)v(t)\ \! d\nu_s(t)\neq 0.
$$
This contradiction shows that $q_{k,s}$ has exactly $k$ distinct real zeros. Then they all must be simple. Since $q_{k,s}$ is even or odd, its zeros are symmetric about the origin.

Assume now in the contrary that $q_{k,s}$ has a root $\alpha$ such that $\alpha^2\geq s^2$. Then $-\alpha$ is also a root of $q_{k,s}$ and $k\geq 2$.
The polynomial $w(t):=\frac {q_{k,s}(t)}{t^2-\alpha^2}$ has degree $k-2$ and using the orthogonality property of $q_{k,s}$ we obtain that
\begin{equation*}
\begin{split}
0&=\int\limits_{-1}^{1}q_{k,s}(t)w(t)\ \! d\nu_s(t)=\int\limits_{-1}^{1}(t^2-\alpha^2)(w(t))^2\ \! d\nu_s(t)\\
&=\int\limits_{-1}^{1}(t^2-s^2)(w(t))^2\ \! d\nu_s(t)+(s^2-\alpha^2)\int\limits_{-1}^{1}(w(t))^2\ \! d\nu_s(t)\\
&=-\int\limits_{-1}^{1}(t^2-s^2)^2(w(t))^2\ \! d\mu_n(t)+(s^2-\alpha^2)\int\limits_{-1}^{1}(w(t))^2\ \! d\nu_s(t)<0,
\end{split}
\end{equation*}
where the negativity of the sum of integrals follows from the fact that $\mu_n$ is a positive measure, the signed measure $\nu_s$ is positive definite up to degree $k-1$, and the contrary assumption that $\alpha^2\geq s^2$.
This contradiction shows that all zeros of $q_{k,s}$ are contained in the interval $(-s,s)$.
\end{proof}

Next, denote $\lambda_0:=-s$ and $\lambda_{k+1}:=s$ and let the weights $\{ \theta_i \}_{i=0}^{k+1}$ be defined by 
\begin{equation}\label{weights_beta''} 
\theta_i=\gamma_{n}\int_{-1}^1 \ell_i (t) (1-t^2)^{(n-3)/2}\, dt ,
%\left(=\gamma_{n}\int_{-1}^1 \ell_i^2 (t) (1-t^2)^{(n-3)/2}\, dt\right) 
\end{equation} 
where $\ell_i(t)$ are the fundamental Lagrange polynomials associated with the nodes $\{ \lambda_i \}_{i=0}^{k+1}$. Note that for $s=1$, the density of $\nu_s$ equals modulo a constant factor the Gegenbauer weight corresponding to the dimension $n+2$. Then the polynomial $q_{k,1}$ equals $P_k^{(n+2)}$ modulo a constant factor. Therefore, in the case $s=1$, we have $\lambda_i=\beta_i$ and $\theta_i=\delta_i$, $i=0,\ldots,k+1$.

\begin{lemma}\label{signed}
For any $n\geq 2$, $k\geq 1$, and $\alpha_{k+1}<s\leq 1$, the quadrature formula
\begin{equation} \label{GJs}
\gamma_{n} \int_{-1}^1 f(t) (1-t^2)^{(n-3)/2} \, dt\approx\sum_{i=0}^{k+1} \theta_i f(\lambda_i) 
%\nonumber
\end{equation}
is exact on $\Pi_{2k+1}$. Furthermore, $\theta_i>0$ and $\theta_i=\theta_{k+1-i}$, $i=0,\ldots,k+1$.
\end{lemma}
\begin{proof}
Let $f$ be any polynomial of degree at most $2k+1$. There exist polynomials $u$ and $v$ of degrees
at most $k-1$ and $k+1$, respectively, such that
$$f(t)=(s^2-t^2)q_{k,s}(t)u(t)+v(t).$$
The $\nu_s$-orthogonality of $q_{k,s}$ to $u(t)$ yields that
$$\int_{-1}^1 f(t) d \mu_n(t)=\int_{-1}^1 v(t) d \mu_n(t).$$
As $(s^2-t^2)q_{k,s}(t)$ vanishes at each node of quadrature formula \eqref{GJs} we have that 
$$\sum_{i=0}^{k+1} \theta_i f(\lambda_i) =\sum_{i=0}^{k+1} \theta_i v(\lambda_i) .$$
Finally, since the quadrature formula \eqref{GJs} is interpolatory, it is exact on $\Pi_{k+1}$ and, hence, for $v$. We then infer that it holds for $f$ as well.

We next establish the positivity of the weights $\theta_i$. Recall that the roots of $q_{k,s}$ are symmetric about the origin (with $0$ being one of the roots when $k$ is odd). This easily implies the symmetry of the weights $\theta_i$, namely that $\theta_j=\theta_{k+1-j}$, $j=0,1,\dots,k+1$. Using $f=q_{k,s}^2$ in \eqref{GJs} we obtain
\begin{eqnarray*}
0<\gamma_{n} \int_{-1}^1 q_{k,s}^2(t) (1-t^2)^{\frac{n-3}{2}} \, dt=\theta_0q_{k,s}^2(-s)+\theta_{k+1}q_{k,s}^2(s)=2\theta_0q_{k,s}^2(-s),
\end{eqnarray*}
which yields that $\theta_0=\theta_{k+1}>0$. Using $f(t)=(s^2-t^2)[q_{k,s}(t)/(t-\lambda_i)]^2$, $i=1,\dots,k$,  we conclude utilizing Lemma \ref{SignedNu_s} and quadrature \eqref{GJs} that
\begin{eqnarray*}
0<\int_{-1}^1 \left[ \frac{q_{k,s}(t)}{t-\lambda_i}\right]^2\, d\nu_s(t)=\int\limits_{-1}^{1}f(t)\ \! d\mu_n(t)=\sum\limits_{j=0}^{k+1}\theta_jf(\lambda_j)=\theta_if(\lambda_i)=\theta_i(s^2-\lambda_i^2)[q_{k,s}^\prime(\lambda_i)]^2,
\end{eqnarray*}
$i=1,\ldots,k$.
Hence, the positivity of the weights is established. 
\end{proof}

\section{Main Results}\label{main}

Ambrus and Nietert \cite {AmbNie2019} proved the min-max polarization optimality of $(1,1)$-designs for the potential $h(t)=t^2$ (among arbitrary codes on $\mathbb{S}^{n-1}$ of the same cardinality). The energy minimizing property of $(1,1)$-designs with respect to this potential $h$ was established by Benedetto and Fickus \cite{BF03}. For the potential $h(t)=t^{2k}$, the energy minimizing property of $(k,k)$-designs was established in the works by Sidel'nikov \cite{Sid1974}, Venkov \cite{Ven2001}, and Welch \cite{Wel2006}. In \cite{BGMPV1,BGMPV2} Bilyk et. al. studied energy minimization problems for the more general {\em $p$-frame} potentials $h(t)=|t|^p$, $p>0$.

In \cite{BDHSS-22a,BDHSS-22b} universal polarization bounds for spherical designs were derived (see also \cite{B,Bor-new,Bor-stiff}). In the current article we adopt our previous approach with a view toward applications to tight frames and $p$-frame potentials. Motivated by these applications we restrict ourselves to even potentials $h$ and relax the differentiability conditions imposed on $h$ in \cite{BDHSS-22a}. To state our universal bounds we shall set for the remainder of the paper \begin{equation}\label{f_g}
    h(t):=g(t^2), \quad g:[0,1]\to (-\infty,\infty],
\end{equation} 
where $g$ is finite and continuous on $[0,1)$, continuous in the extended sense at $t=1$, and $(k+1)$-differentiable in $(0,1)$ for some $k\in \mathbb{N}$. Examples of such potentials include the $p$-frame potential $h(t)=|t|^p$, the symmetric Riesz potential $h(t)=(1+t)^{-p}+(1-t)^{-p}$, the arcsine potential $h(t)=1/\sqrt{1-t^2}$, and the symmetric Gaussian potential $h(t)=\cosh(t)$.

The universal bounds we obtain here hold for each $(k,k)$-design (and thus are universal in the Levenshtein sense \cite{Lev98}), and are valid for a large class of potential functions (and hence in the Cohn-Kumar sense \cite{CK}). Moreover, the bounds are computed at certain nodes and with 
certain weights that are independent of the potential. 
In the special case of $h(t)=t^{2k}$, we shall prove 
that the $(k,k)$-designs are optimal (for the polarization problems under consideration) in the class of all spherical codes of the same cardinality (see Theorem \ref{minmax-maxmin-extract}).

We next state the main results of this article. Given the set $-1<\alpha_1<\ldots<\alpha_{k+1}<1$ of zeros of the Gegenbauer polynomial $P_{k+1}^{(n)}$, denote by $H_{2k}$ the unique polynomial of degree at most $2k$ such that for every $i=1,\ldots,k+1$, we have $H_{2k}(\alpha_i)=h(\alpha_i)$ and, if $\alpha_i\neq 0$, also $H_{2k}'(\alpha_i)=h'(\alpha_i)$. Since $h$ is even, these interpolation conditions are symmetric about $0$. Therefore, $H_{2k}$ is even. If $k$ is odd, then $H_{2k}$ interpolates $h$ at each $\alpha_i$ to degree two and is called the Hermite interpolation polynomial for $h$ at the system of nodes $\{\alpha_1,\ldots,\alpha_{k+1}\}$, see, e.g., \cite[p. 330]{Sze1975} or \cite[Section 2.2]{IsaKel1994}. Since it satisfies $2k+2$ interpolation conditions, one must look for it in $\Pi_{2k+1}$. However, since $H_{2k}$ is even, it must be in $\Pi_{2k}$. If $k$ is even, then $H_{2k}$ interpolates $h$ at $t=0$ to degree one. However, if $h$ is differentiable at $t=0$, then $h'(0)=H'_{2k}(0)=0$ because both $h$ and $H_{2k}$ are even; that is $H_{2k}$ will be a Hermite again. Therefore, we will call $H_{2k}$ the Hermite interpolation polynomial for $k$ of both parities, keeping in mind that when $k$ is even and $h'(0)$ does not exist, $H_{2k}$ interpolates $h$ at $t=0$ only to degree one. 

Given the set $-1=\beta_0<\beta_1<\ldots<\beta_k<\beta_{k+1}=1$, where $\beta_i$, $i=1,\ldots,k$, are zeros of the Gegenbauer polynomial $P_{k}^{(n+2)}$, denote by $\W H_{2k}$ the unique polynomial of degree at most $2k$ such that for every $i=0,\ldots,k+1$, we have $\W H_{2k}(\beta_i)=h(\beta_i)$ and, if $1\leq i\leq k$ and $\beta_i\neq 0$, also $\W H_{2k}'(\beta_i)=h'(\beta_i)$. Since $h$ is even and the interpolation conditions are again symmetric about the origin, $\W H_{2k}$ is even and it is sufficient to search for it in $\Pi_{2k}$. Like in the previous paragraph, we will still call $\W H_{2k}$ the Hermite interpolation polynomial, keeping in mind that $\W H_{2k}$ interpolates $h$ to degree one at $t=1$, $t=-1$, and, if $k$ is odd and $h'(0)$ does not exist, at $t=0$.  
\begin{theorem}[Lower bounds]\label{PULB-extract}
Assume there exists a spherical $(k,k)$-design $C$ of cardinality $N$ on $\mathbb{S}^{n-1}$ and 
that the potential $h$ is as in \eqref{f_g}. Then

\begin{itemize}
    \item[(i)] if $g^{(k+1)}(u)\geq 0$ on $(0,1)$, the following bounds hold
\begin{equation}\label{PolarizationULB-e1}
m_{N}^h \geq m^h_{N}(k)\geq m^h(C) \geq N \cdot \sum_{i=1}^{k+1} \rho_i h(\alpha_i), 
\end{equation}
where $\{ \alpha_i \}_{i=1}^{k+1}$ are the roots of the Gegenbauer polynomial $P_{k+1}^{(n)}$ and 
$\{ \rho_i \}_{i=1}^{k+1}$ are associated weights {\rm (}defined in \eqref{rho_weights'}{\rm )}. Moreover, the bound \eqref{PolarizationULB-e1} 
is the best that can be obtained via polynomials from the set ${\mathcal L}(k;h)$ {\rm (}see \eqref{L(k,h)}{\rm )} and is provided by the unique maximizer $H_{2k}(t)$ of the linear program \eqref{PULB_LP_Program}; that is the Hermite interpolation polynomial to $h$ at the nodes $\{\alpha_i \}_{i=1}^{k+1}$ as described in the paragraphs above.
 \item[(ii)] if $g^{(k+1)}(u)\leq 0$ on $(0,1)$, the following bounds hold
\begin{equation}\label{PolarizationULB-e2}
m_{N}^h \geq m^h_{N}(k)\geq m^h(C) \geq N \cdot \sum_{i=0}^{k+1} \delta_i h(\beta_i), 
\end{equation}
where $\beta_0=-1$, $\beta_{k+1}=1$, and $\{ \beta_i \}_{i=1}^{k}$ are the roots of the Gegenbauer polynomial $P_{k}^{(n+2)}$ and 
$\{ \delta_i \}_{i=0}^{k+1}$ are associated weights {\rm (}defined in \eqref{delti-1}-\eqref{delti-3}{\rm )}. Moreover, the bound \eqref{PolarizationULB-e2} 
is the best that can be obtained via polynomials from the set ${\mathcal L}(k;h)$ {\rm (}see \eqref{L(k,h)}{\rm )} and is provided by the unique maximizer $\W H_{2k}(t)$ of the linear program \eqref{PULB_LP_Program}; that is the Hermite interpolation polynomial to $h$ at the nodes $\{ \beta_i \}_{i=0}^{k+1}$ as described in the paragraph above.
\end{itemize}
\end{theorem}

\begin{remark}
    Note that the condition $g^{(k+1)}(u)\leq 0$ on $(0,1)$ in Theorem \ref{PULB-extract}(ii) along with the continuity of $g$ in the extended sense on $[0,1]$ implies that $g(1)$ is finite (see the proof).
\end{remark}

As an immediate consequence of Theorem \ref{PULB-extract}(ii),  
by considering the potential $-h$, we obtain Theorem \ref{PUUB-extract} giving a universal upper bound for $M_N^h (k)$. An alternative proof of Theorem \ref{PUUB-extract} based on Theorem \ref{PUUB-extract-s}, is provided in Section \ref{section-UUB}. 

\begin{theorem}[Upper bounds -- finite case]\label{PUUB-extract} 
Assume there exists a spherical $(k,k)$-design $C$
of cardinality $N$ on $\mathbb{S}^{n-1}$ and 
that the potential $h$ is as in \eqref{f_g}. Then if $g^{(k+1)}(u)\geq 0$ on $(0,1)$ and $g(1)$ is finite, the following bounds hold
\begin{equation}\label{PolarizationUUB-e}
M^h_{N}\leq M^h_{N}(k) \leq M^h(C)\leq N \cdot \sum_{i=0}^{k+1} \delta_i h(\beta_i), 
\end{equation}
where the parameters $\{ \beta_i \}_{i=0}^{k+1}$ and 
$\{ \delta_i \}_{i=0}^{k+1}$ are as in Theorem \ref{PULB-extract}(ii). 
\end{theorem}
The solution to linear program \eqref{PUUB_LP_Program} in Theorem \ref{PUUB-extract} follows from the solution to linear program \eqref{PULB_LP_Program} in Theorem \ref{PULB-extract} (ii).

When $g(1)=\infty$ we utilize the notion of a positive definiteness up to a certain degree for signed measures (see Definition \ref{nu_pos}) to find a universal upper bound on $m_N^h (k)$.

\begin{theorem}[Upper bounds -- infinite case]\label{PUUB-extract-s} 
Assume there exists a spherical $(k,k)$-design $C$
of cardinality $N$ on $\mathbb{S}^{n-1}$, and let $R_{k,N}<s<1$, where $R_{k,N}$ is defined in \eqref{R_{k,N}}.
Suppose the potential $h$ is as in \eqref{f_g} and that $g^{(k+1)}(u)\geq 0$ on $(0,s^2)$.
Then the following bound holds:
\begin{equation}\label{PolarizationUUB-e3}
m^h(C) \leq m^h_{N}(k)\leq  N \cdot \sum_{i=0}^{k+1} \theta_i h(\lambda_i), 
\end{equation}
where 
$-\lambda_0=\lambda_{k+1}=s$
and 
$\{\lambda_i\}_{i=1}^{k}$ 
are the zeros of the $k$-th degree polynomial $q_{k,s}$, orthogonal to all lower degree polynomials with respect to the signed measure $\nu_s$ {\rm (}see \eqref{nu_s'}{\rm )}, and $\{ \theta_i \}_{i=0}^{k+1}$ are associated weights {\rm (}defined in \eqref{weights_beta''}{\rm{})}.
\end{theorem}
In the proof of Theorem \ref{PUUB-extract-s}, the unique solution to linear program \eqref{PUUB_LP_Program} is also found.
Observe that in Theorem \ref{PUUB-extract-s}, the value $h(1)$ can be infinite. If $h(1)< \infty$,  then Theorem \ref{PUUB-extract} may be derived from the proof of Theorem \ref{PUUB-extract-s} by letting $s\to 1^-$ (see Remark \ref{Rem_1}).

In an important special case, we prove the optimality of $(k,k)$-designs whenever they exist both for min-max and max-min polarization. 

\begin{theorem}\label{minmax-maxmin-extract}
Let $n \geq 2$, $k,N\geq 1$, and $h(t)=t^{2k}$. Then for every $N$-point code $C\subset \mathbb{S}^{n-1}$,
\begin{equation}\label{m1-m2}
m{}^{h}(C)\leq c_{2k}N \leq M{}^{h}(C),
\end{equation}
where $c_\ell$ is defined by \eqref{cell}. 
Each inequality in \eqref{m1-m2} becomes an equality if and only if 
$C$ is an $N$-point spherical $(k,k)$-design on $\mathbb{S}^{n-1}$. Moreover, if an $N$-point spherical $(k,k)$-design exists on $\mathbb{S}^{n-1}$, we have
$$ m^h_{N}=M^h_{N}=c_{2k}N $$
for $h(t)=t^{2k}$.
\end{theorem}
Theorem \ref{minmax-maxmin-extract} follows immediately by combining Theorems \ref{maxmin} and \ref{minmax} established in Section \ref{kk-optimality}.

\section{Universal lower bounds -- proofs, examples, and applications} \label{Ubounds}

In this section, we will derive polarization lower bounds (PULB) for $(k,k)$-designs. In doing so, we first show that the interpolating polynomial $H_{2k}$ (or $\W H_{2k}$) stays below the potential function $h$. Then we prove the required lower bound for the potential of the code $C$ using the fact that it is a $(k,k)$-design and the exactness property of the corresponding quadrature. Using these tools, we then show the optimality of the polynomial $H_{2k}$ ($\W H_{2k}$) for linear program \eqref{PULB_LP_Program} and its uniqueness as the maximizer.

Recall that $h$ satisfies \eqref{f_g}, namely $h(t)=g(t^2)$, $g:[0,1]\to (-\infty,\infty]$, where $g$ is continuous on $[0,1)$, continuous in the extended sense at $t=1$, and $(k+1)$-differentiable in $(0,1)$. Considering $h$ in terms of $g$ allows us to utilize the properties of $(k,k)$-designs and relax the differentiability requirements for $h$.

%We will start by providing a solution to linear program \eqref{PULB_LP_Program} for such potential functions $h$.

\subsection{Proof of Theorem \ref{PULB-extract}(i)}
 
We shall start by proving the last inequality in \eqref{PolarizationULB-e1},  
since the existence of a spherical $(k,k)$-design $C$ of cardinality $N$ on $\mathbb S^{n-1}$ implies
that $m_N^h \geq m_N^h(k)\geq m^h(C)$. 

Recall that the zeros $-1<\alpha_1<\ldots<\alpha_{k+1}<1$ of the Gegenbauer polynomial $P_{k+1}^{(n)}$ are symmetric about $0$ (see subsection \ref{sGJQ}). Our considerations will depend on the parity of $k+1$. Therefore, we let $k=2\ell-1+\mu$, $\ell\in \mathbb Z$ and $\mu\in \{0,1\}$, and define numbers $0\leq u_{1-\mu}<\ldots<u_\ell<1$ by $$-\alpha_i=\alpha_{k+2-i}=\sqrt{u_{\ell+1-i}}, \ \ \ i=1,\ldots,\ell+\mu.$$

Let $G_{k}(u)$ be the unique Hermite interpolation polynomial for $g$ at the nodes 
$\{ u_i \}_{i=1-\mu}^{\ell}$, where when $\mu=1$ we apply Lagrange interpolation at the node $u_0=0$ (recall that $g$ need not be differentiable at $0$), namely $G_k\in \Pi_k$, $G_{k}(u_i)=g(u_i)$, $i=1-\mu,\dots,\ell$ 
and 
$G_{k}^\prime(u_i)=g^{\prime} (u_i)$, $i=1,\dots,\ell$. The Hermite interpolation error formula (see \cite[Section 2.2]{IsaKel1994} or \cite[Theorem 3.5.1]{Dav1975}) implies that for any $u\in [0,1)$, there exists $\xi=\xi(u)\in (0,1)$ such that
\begin{equation}\label{HermiteError}
g(u)-  G_{k}(u) = \frac{g^{(k+1)}(\xi)}{(k+1)!}u^\mu(u-u_1)^2 \dots (u-u_{\ell})^2\geq 0 .\end{equation}
Therefore, $G_k (u)\leq g(u)$ on $[0,1]$ (recall the extended continuity of $g$ at $1$).

Next, we show that $G_k(t^2)=H_{2k}(t)$. Indeed, for every $i=1,\dots,\ell+\mu$ we have 
$$H_{2k}(\alpha_i)=h(\alpha_i)=g(\alpha_i^2)=g(u_{\ell+1-i})=G_k(u_{\ell+1-i})=G_k(\alpha_i^2)$$
and for every $i=1,\dots,\ell$ we have
$$ H_{2k}'(\alpha_i)=h'(\alpha_i)=2\alpha_i g'(\alpha_i^2)=-2\sqrt{u_{\ell+1-i}}g'(u_{\ell+1-i})=-2\sqrt{u_{\ell+1-i}} G_k'(u_{\ell+1-i}) =2\alpha_i G_k'(\alpha_i^2).$$
Since both $H_{2k}(t)(t)$ and $G_k(t^2)$ are even, we conclude that $G_{k}(t^2)$ (which is in $\Pi_{2k}$) interpolates $h$ to degree two at every non-zero $\alpha_i$, $i=1,\ldots,k+1$. If $\mu=1$ it interpolates $h$ at the origin and if $h$ is differentiable at $0$, since it is an even function we have that $h^\prime (0)=0$ and so $\frac{d}{dt}G_k(t^2)|_{t=0}=h^\prime (0)$ as well. These interpolation conditions together with the fact that $G_k(t^2)$ has degree at most $2k$ determine $G_k(t^2)$ uniquely; that is, $G_k(t^2)=H_{2k}(t)$.

We now clearly have $H_{2k}(t)\leq h(t)$ on $[-1,1]$, or $H_{2k}\in {\mathcal L}(k;h)$. If $C$ is any $(k,k)$-design on $\mathbb S^{n-1}$ of cardinality $N$, for any $x\in \mathbb S^{n-1}$, we derive that
\begin{equation}\label {26a}
\begin {split}
U^h(x,C)&=\sum\limits_{i=1}^{N}h(x\cdot x_i)\geq \sum\limits_{i=1}^{N}H_{2k}(x\cdot x_i) \\
  &= N(H_{2k})_0= N\gamma_{n} \int_{-1}^1 H_{2k}(t) (1-t^2)^{(n-3)/2}\, dt \\
                           &= N \cdot \sum_{i=1}^{k+1} \rho_i H_{2k}(\alpha_i) = N \cdot \sum_{i=1}^{k+1} \rho_i h(\alpha_i),
\end {split}
\end{equation}
where we used the fact that $h\geq H_{2k}$, Lemma \ref{k-k-lemma} v), the exactness of quadrature \eqref{GJQ} on polynomials of degree up to $2k+1$ (see Lemma \ref{LemmaGJQ}), and the fact that $H_{2k}$ interpolates $h$ at the nodes $\alpha_i$. The inequality $U^h(x,C)\geq N(H_{2k})_0$, $x\in \mathbb S^{n-1}$, also follows directly from Theorem~\ref{PULB-k-k-general}. Then
$$
m_N^h \geq m_N^h(k) \geq  m^h(C)=\inf\limits_{x\in \mathbb S^{n-1}}U^h(x,C)\geq N \cdot \sum_{i=1}^{k+1} \rho_i h(\alpha_i),
$$
which proves \eqref{PolarizationULB-e1}. 

To show the optimality of $H_{2k}$ for the linear program \eqref{PULB_LP_Program}, let $f$ be an arbitrary (even) polynomial of degree at most $2k$ in the class ${\mathcal L}(k;h)$. Using Lemma \ref{LemmaGJQ} and \eqref{26a} we have 
\begin{equation}\label{LPH2k} Nf_0=N\gamma_{n} \int_{-1}^1 f(t) (1-t^2)^{(n-3)/2}\, dt = N \cdot \sum_{i=1}^{k+1} \rho_i f(\alpha_i)\leq N \cdot \sum_{i=1}^{k+1} \rho_i h(\alpha_i)=N(H_{2k})_0,
\end{equation}
which implies $f_0\leq (H_{2k})_0$; that is, $H_{2k}$ is optimal.
To show the uniqueness of the maximizer,
we recall that $\rho_i>0$, $i=1,\ldots,k+1$, see Lemma \ref{LemmaGJQ}.
For equality to hold in \eqref{LPH2k}, we need $f(\alpha_i)=h(\alpha_i)$ for $i=1,\dots,k+1$. Let $g_k \in \Pi_k$ be such that $g_k (t^2)=f(t)$. Then $g_k (u_i)=g(u_i)$, $i=1-\mu,\dots,\ell$, which along with $g_k (u)\leq g(u)$ on $[0,1]$ implies that $g'_k (u_i)=g' (u_i )$, $i=1,\dots,\ell$. These interpolation conditions and the fact that $g_k\in \Pi_k$ determine $g_k$ uniquely; that is, $g_k=G_k$, and hence $f=H_{2k}$, which completes the proof.
\hfill $\Box$

\subsection{Proof of Theorem \ref{PULB-extract}(ii)} In this case we follow similar approach with some modification. We first show that $g(1)$ is finite. Indeed, let $T_k(t;g)$ denote the degree $k$ Taylor polynomial of $g$ at $t=1/2$. Then the integral error formula and the assumption that $g^{(k+1)}\leq 0$ on $(0,1)$ yield
$$g(t)=T_k (t;g)+\int_{1/2}^t g^{(k+1)}(u)\frac{(t-u)^k}{k!}\, du\leq \max_{u\in [1/2,1]} T_k(u;g), \ \frac{1}{2}<t<1; $$
that is, $g(t)$ is bounded above on $(1/2,1)$. Then the claim that $g(1)<\infty$ follows from the continuity of $g$ in the extended sense at $t=1$.

Let $k=2\ell+\mu$, $\mu\in \{0,1\}$, $\ell\in \mathbb Z$. 
%The symmetry of the weights $\delta_i=\delta_{k+1-i}$ implies that both sides of \eqref{GJQ_beta'} are zero for all odd polynomials of degree at most $2k+1$. For even polynomials of degree at most $2k$ we follow the same approach, where the relation with the quadrature (and interpolation) nodes $\{ u_i\}$ on $[0,1]$ is given by $P_k^{(n+2)}(t)=t^\mu p_\ell^\mu(t^2)$, where $p_\ell^\mu (u)$ is orthogonal to all polynomials of degree at most $\ell-1$ with orthogonality measure $\omega_\mu (u)\, du := \gamma_{n}(1-u)^{(n-1)/2}u^{-1/2+\mu}\, du$ and normalization $p_\ell^\mu (1)=1$.
Let $G_k$ be the unique polynomial of degree at most $k$ that interpolates $g$ and $g^\prime$ at the nodes $u_i=(\beta_{k+1-i})^2$, $i=1,\dots,\ell$, and that interpolates $g$ at $u=1$, and when $\mu=1$ also at $u=0$.

A Hermite interpolation error formula, similar to \eqref{HermiteError}, yields that for any $u\in(0,1)$, there exists $\xi=\xi(u) \in (0,1)$, such that  (see, e.g., \cite[Section 2.2]{IsaKel1994} or \cite[Theorem 3.5.1]{Dav1975})
%\begin{equation}\label{HermiteError2}
$$
g(u)-  G_{k}(u) = \frac{g^{(k+1)}(\xi)}{(k+1)!}u^\mu(u-u_1)^2 \dots (u-u_{\ell})^2(u-1).$$
%\end{equation}
Because $g^{(k+1)}\leq 0$, we conclude that $G_k(u)\leq g(u)$ for all $u\in (0,1)$ and from the continuity of $g$ on $[0,1]$ we obtain that $G_k(u)\leq g(u)$, $u\in [0,1]$.

We derive in a similar manner to the proof of (i) that $G_k(t^2)$ coincides with the polynomial $\W H_{2k}(t)$, which is defined before Theorem \ref{PULB-extract} and interpolates $h$ and $h^\prime$ at any $\beta_i \in (-1,0)\cup (0,1)$, interpolates $h$ at $t=-1$, at $t=1$, and when $\mu=1$ at $t=0$. Then $\W H_{2k}(t)\leq h(t)$, $t\in[-1,1]$; that is, $\W H_{2k}\in \mathcal{L}(k;h)$. For any $(k,k)$-design $C$ on $\mathbb S^{n-1}$ of cardinality $N$, taking also into account Lemma \ref{k-k-lemma} v) and Lemma \ref{LemmaGJQ1} we derive for any $x\in \mathbb S^{n-1}$ that
\begin{equation}\label {26a-2}
\begin{split}
U^h(x,C)&=\sum_{i=1}^{N}h(x\cdot x_i)\geq \sum_{i=1}^{N}\W {H}_{2k}(x\cdot x_i)=N(\W {H}_{2k})_0=N\gamma_{n} \int_{-1}^1 \W {H}_{2k}(t) (1-t^2)^{(n-3)/2}\, dt \\
                           &= N \cdot \sum_{i=0}^{k+1} \delta_i \W {H}_{2k}(\beta_i) = N \cdot \sum_{i=0}^{k+1} \delta_i h(\beta_i),
\end{split}
\end{equation}
which completes the proof of \eqref{PolarizationULB-e2}. The inequality $U^h(x,C)\geq N(\W H_{2k})_0$, $x\in \mathbb S^{n-1}$, also follows directly from Theorem \ref{PULB-k-k-general}.

The optimality of $\W H_{2k}$ and its uniqueness as a maximizer are shown analogously to the proof of (i).\hfill $\Box$

\begin{remark}
    From \eqref {26a}, we also obtain that
$$M_N^h(k) \geq N \cdot \sum_{i=1}^{k+1} \rho_i h(\alpha_i)
$$
with parameters as in Theorem \ref{PULB-extract}(i). 
Similarly, from \eqref{26a-2} we conclude that
%\begin{equation}\label{PolarizationULB_R-e2}
$$
M_{N}^h(k) \geq N \cdot \sum_{i=0}^{k+1} \delta_i h(\beta_i),
$$
%\end{equation}
where the parameters are as in Theorem \ref{PULB-extract}(ii). 
\end{remark}

\subsection{Examples and applications} 

\begin{example}[\bf Stiff codes]
A spherical code is {\em stiff} (cf. \cite{Bor-stiff}) if it is a spherical $(2k-1)$-design contained in $k$ parallel hyperplanes, i.e. there is a point on $\mathbb{S}^{n-1}$ whose inner products with points of the code take exactly $k$ distinct values. The 
antipodal stiff codes give rise to a number of codes attaining the bound \eqref{PolarizationULB-e1}. Indeed, assume that $C \subset \mathbb{S}^{n-1}$
is an antipodal stiff code. Then the code formed by choosing a point from each pair of antipodal points of $C$ is a spherical $(k-1,k-1)$-design (cf. \cite{BZ4})
which attains \eqref{PolarizationULB-e1}. 
Note that there are $2^{|C|/2}$ such codes (some of which will be isometric); for example the cube 
provides $2^4=16$ optimal $(1,1)$-designs on $\mathbb{S}^2$ and the $24$-cell provides $2^{12}=4096$ optimal $(2,2)$-designs on $\mathbb{S}^{3}$. 
\end{example}

\begin{example}[\bf Unit-norm tight frames]
    Recall that a $(1,1)$-spherical design $C$ is called a unit-norm tight frame. Let us first consider $p$-frame potentials $h(t)=\varepsilon|t|^p$, $\varepsilon\in \{\pm 1\}$, $p\geq 2$. The corresponding potential functions on $[0,1]$ are $g(u)=\varepsilon u^{p/2}$. When $\varepsilon=1$, $g ''(u)=(p/2)(p/2-1)u^{p/2-2}\geq 0$ on $(0,1)$ and bound \eqref{PolarizationULB-e1} holds with $k=1$. In this case, $P_2^{(n)}(t)=\frac {nt^2-1}{n-1}$, $\alpha_1=-\frac {1}{\sqrt{n}}$, $\alpha_2=\frac {1}{\sqrt{n}}$, and by symmetry of quadrature \eqref{GJQ} and its exactness on constants, we have $\rho_1=\rho_2=\frac {1}{2}$. Consequently,
    $$\sum_{y \in C} |x\cdot y|^p \geq \frac{N}{n^{p/2}}.$$

    When $\epsilon=-1$, then the bound \eqref{PolarizationULB-e2} holds as $g''(u)\leq 0$. We may compute directly (recall that $P_1^{(n+2)}(t)=t$) that $\delta_0=\delta_2=1/(2n)$ and $\delta_1=(n-1)/n$, and that $\beta_0=-1$, $\beta_1=0$, and $\beta_2=1$. Thus, 
     $$\sum_{y\in C} -|x\cdot y|^p \geq -\frac{N}{n}.$$
     Observe, we obtain the same bound if we apply Theorem \ref{PUUB-extract} for the potential $h(t)=|t|^p$. 
     
     Considering the case $0<p<2$ in a similar way we conclude the following corollary for tight frames. 
     \begin{corollary}
     For any unit-norm tight frame $C\subset \mathbb{S}^{n-1}$, the following estimates on its $p$-frame potential hold
    \begin{equation}\label{W21}
       \frac{|C|}{n}\geq \sum_{y\in C} |x\cdot y|^p \geq \frac{|C|}{n^{p/2}}, \ \ \ p\geq 2,
   \end{equation}
   and 
   \begin{equation}\label{W21ab}
     \frac{|C|}{n^{p/2}}  \geq \sum_{y\in C} |x\cdot y|^p \geq \frac{|C|}{n}, \ \ \ 0<p<2.
   \end{equation}
     \end{corollary}
     
\begin{remark}
    Note that a Waring-type identity follows immediately when one sets $p=2$. Indeed, for $p=2$, inequalities \eqref{W21} yield the equality $\sum\limits_{y\in C}(x\cdot y)^2=\frac {|C|}{n}$, which is a special case of \eqref{Waring} with $\ell=2$ and $\|x\|=1$.
\end{remark}
\end{example}

As a consequence of Theorem  \ref{PULB-extract}(i) we prove a Fazekas-Levenshtein-type bound (see \cite{FL}) for the covering radius of spherical $(k,k)$-designs. 
It is worth to mention that, utilizing again the relation between the $(k,k)$-designs and 
antipodal spherical $(2k+1)$-designs, the
Fazekas-Levenshtein bound in the case of real projective space also follows.

\begin{corollary} \label{FL-type}
Let $k\geq 1$ be an integer and suppose there is a spherical $(k,k)$-design $C\subset \mathbb{S}^{n-1}$ with $|C|=N$. Then $r_{k,N}\geq \alpha_{k+1}$.
\end{corollary}

\begin{proof} 
Assume to the contrary that $r_{k,N}< \alpha_{k+1}$. Then there exists a spherical $(k,k)$-design $C$ of cardinality $|C|=N$ such that $r(C)<\alpha_{k+1}$. Theorem \ref{PULB-extract} (i) 
implies that $$ \sum_{y\in C}h(x\cdot y)\geq m^h(C)\geq N \cdot \sum_{i=1}^{k+1} \rho_i h(\alpha_i)$$
for all even potentials $h$ as in \eqref{f_g} with $g^{(k+1)}\geq 0$ on $(0,1)$, in particular, for the Riesz-type potentials 
$$h(t)=h_m (t):=(2-2t)^{-m/2}+(2+2t)^{-m/2}, \ m>0.$$  
Indeed, the function $R(t):=(2-2t)^{-m/2}$ is strictly absolutely monotone on $(-1,1)$ and hence, its Taylor series at $t=0$ has positive coefficients. Then the Taylor series of $h_m(t)$ at $t=0$
has only even powers of $t$ with positive coefficients. That is, $h_m(t)=g(t^2)$ for some strictly absolutely monotone $g$ on $(0,1)$.
Therefore, $h_m (t)$  strictly decreases on $(-1,0]$ and strictly increases on $[0,1)$. There is $x \in \mathbb{S}^{n-1}$ such that $T(x,C\cup -C)\subset [-r(C),r(C)]$. For such a $x$, since $\rho_i>0$, $i=1,\dots,k+1$, see Lemma \ref{LemmaGJQ}, we have that
\begin{eqnarray*}
\frac{N}{(2-2r(C))^{m/2}}+\frac{N}{(2+2r(C))^{m/2}} &\geq& \sum_{y\in C}h_m (x\cdot y) \geq N \cdot \sum_{i=1}^{k+1} \rho_i h_m(\alpha_i) \\
&\geq& \frac{N\rho_{k+1}}{(2-2\alpha_{k+1})^{m/2}}+\frac{N\rho_{k+1}}{(2+2\alpha_{k+1})^{m/2}}.
\end{eqnarray*}
Taking an $m$-th root and letting $m\to \infty$ we derive a contradiction.
\end{proof}

\section {Universal upper bounds} \label{section-UUB}

\subsection{Upper bounds}

We shall utilize the Fazekas-Levenshtein-type bound from Corollary \ref{FL-type} to derive polarization universal upper bounds (PUUB) for $(k,k)$-designs. In analyzing the program \eqref{PUUB_LP_Program} 
we shall use polynomials $q_{k,s}$ obtained in subsection \ref{PDmeasures} via the Gram-Schmidt orthogonalization process with respect to 
the signed measure $\nu_s$ defined in \eqref{nu_s'}.

Given a $(k,k)$-design $C$, we have that $r(C) \geq r_{k,N}\geq \alpha_{k+1}$
by Corollary \ref{FL-type}. Therefore, fixing any $s>r_{k,N}$ the measure $\nu_s$ is positive definite up to degree $k-1$, see Lemma \ref{SignedNu_s}. Hence, applying Gram-Schmidt orthogonalization procedure we obtain orthogonal polynomials $\{ q_{j,s} (t)\}_{j=0}^{k}$ with respect to $\nu_s$. Note that because of the symmetry of the (signed) measure of orthogonality, the roots of $q_{j,s}$ will also be symmetric about the origin.

Recall that $\{ \lambda_i \}_{i=1}^{k}$ are the $k$ distinct zeros of the polynomial $q_{k,s}(t)$, 
all in the interval $(-s,s)$ as asserted by Lemma \ref{int_s} and that $\lambda_0=-s$ and $\lambda_{k+1}=s$.  
Hence, we have the ordering 
$$ -s=\lambda_0<\lambda_{1}<\cdots<\lambda_{k}<\lambda_{k+1}=s. $$
 We are now in a position to prove Theorem \ref{PUUB-extract-s} establishing a polarization universal upper bound.

\subsection{Proof of Theorem \ref{PUUB-extract-s}} We will use the same approach as in the proof of Theorem \ref{PULB-extract}.
 Applying Corollary \ref{FL-type} we get $\alpha_{k+1}\leq r_{k,N}\leq R_{k,N} < s$.
To find an appropriate polynomial in $\mathcal{U}(k,s;h)$, let $k:=2\ell +\mu$, $\ell\in \mathbb Z$ and $\mu\in \{0,1\}$, and let $G_k(u)$ be the polynomial in $\Pi_k$ that interpolates $g$ and $g'$ at the nodes $u_{\ell+1-i}:=\lambda_i^2$, $i=1,\dots,\ell$, interpolates $g$ at $u_{\ell+1}:=s^2$, and if $\mu=1$ also at $u=0$. The interpolation error formula (see \cite[Section 2.2]{IsaKel1994} or \cite[Theorem 3.5.1]{Dav1975} yields that for any $u\in (0,s^2)$, there exists $\xi=\xi(u)\in (0,s^2)$ such that
\begin{equation} \label{HermiteError3}
g(u)-G_k(u)=\frac{g^{(k+1)}(\xi)}{(k+1)!} u^\mu(u-u_1)^2\dots(u-u_{\ell})^2(u-u_{\ell+1}) \leq 0,
\end{equation}
which implies that $g(u)\leq G_k (u)$ on $[0,s^2]$. Then one obtains that $H_{2k}^s(t):=G_k (t^2) \in \mathcal{U}(k,s;h)$. The interpolation conditions for $G_k$ imply that $H_{2k}^s$ interpolates $h$ at each $\lambda_i$, $i=0,\ldots,k+1$, and $h'$ at each non-zero $\lambda_i$, $i=1,\ldots,k$.
Moreover, since $s>R_{k,N}$, for any spherical $(k,k)$-design $C$, we have that $r(C)< s$.
This means that for some $x^*\in \mathbb{S}^{n-1}$, we have $T( x^*, C \cup -C) \subset (-s, s)$, and hence using Lemma \ref{k-k-lemma} v), Lemma \ref {signed}, the interpolation of $h$ by $H_{2k}^s$, and the inequality $H_{2k}^s(t)\geq h(t)$, $t\in [-s,s]$ we have
\begin{equation}\label{42a}
\begin{split}
m^h(C)&\leq U^h(x^*,C)\leq U^{H_{2k}^s} (x^*,C) =(H_{2k}^s)_0 N \\ &=N \gamma_n\int\limits_{-1}^{1}H_{2k}^s(t)(1-t^2)^{(n-3)/2}\ \! dt=N\cdot \sum\limits_{i=0}^{k+1}\theta_i H_{2k}^s(\lambda_i)= N \cdot \sum_{i=0}^{k+1} \theta_i h(\lambda_i),
\end{split}
\end{equation}
where $(H_{2k}^s)_0$ is the constant coefficient in the Gegenbauer expansion of $H_{2k}^s$. Taking the supremum of $m^h(C)$ over all $(k,k)$-designs of cardinality $N$ yields \eqref{PolarizationUUB-e3}.

To see that the bound is optimal among polynomials in $\mathcal{U}(k,s;h)$, suppose that $f$ is an arbitrary polynomial in the class. Taking into account Lemma \ref{signed} and \eqref{42a}, we then have 
$$ Nf_0=N\gamma_{n} \int_{-1}^1 f(t) (1-t^2)^{(n-3)/2}\, dt = N \cdot \sum_{i=0}^{k+1} \theta_i f(\lambda_i)\geq N \cdot \sum_{i=0}^{k+1} \theta_i h(\lambda_i)=N (H_{2k}^s)_0.$$
Therefore, $f_0\geq (H_{2k}^s)_0$; that is, $H_{2k}^s(t)$ is an optimal polynomial from $\mathcal{U}(k,s;h)$ for linear program \eqref{PUUB_LP_Program}. Its uniqueness is proved analogously to Theorem \ref{PULB-extract} (i) by showing that any other minimizer must satisfy the same interpolation conditions as $H_{2k}^s$ and using the uniqueness of the interpolanting polynomial.
\hfill $\Box$

\begin{remark}\label{Rem_1}
Note that the second inequality in \eqref{42a} holds for any point $z\in \mathbb{S}^{n-1}$ for which $T (z,C\cup -C) \subset [-s,s]$. Taking the supremum over the set of such $z$ (which is the complement of a finite union of spherical caps) and letting $s\to 1^-$ for the case of $h(1)<\infty$, we obtain an alternative proof of Theorem \ref{PUUB-extract}. One needs to use the fact that the value $N\cdot \sum\limits_{i=1}^{N}\theta_ih(\lambda_i)$ is the optimum (minimum) value of linear program \eqref{PUUB_LP_Program} which is an increasing function of $s$. Then it has a limit as $s\to 1^{-}$. This limit equals $N\cdot \sum\limits_{i=1}^{N}\delta_ih(\beta_i)$.
\end{remark}

\subsection{Second proof of Theorem \ref{PUUB-extract}}
We modify the proof of Theorem \ref{PUUB-extract-s} as follows. From \eqref{HermiteError3} with $s=1$ utilizing the fact that $g(1)$ is finite we conclude that $h \leq H_{2k}^s$ on $[-1,1]$. Thus,  
$$  U^h(x,C) \leq U^{H_{2k}^s}(x,C)=N(H^s_{2k})_0=N \cdot \sum_{i=0}^{k+1} \theta_i h(\lambda_i)=N \cdot \sum_{i=0}^{k+1} \delta_i h(\beta_i),$$
where $H_{2k}^s$ is the interpolation polynomial (for $s=1$) to $h$ at the nodes $\{ \lambda_i \}_{i=0}^{k+1}=\{ \beta_i \}_{i=0}^{k+1}$ as described above. Taking the supremum over $x\in\mathbb S^{n-1}$ and the infimum first over all $N$-point $(k,k)$-designs $C$ and then over all $N$-point codes on $S^{n-1}$ we complete the proof.
\hfill $\Box$

\section {Polarization optimality of $(k,k)$-designs} 
\label{kk-optimality}

In this section, we consider the special case of potential $h(t)=t^{2k}$
solving the polarization problem when $(k,k)$-designs of the 
corresponding cardinality exist. To keep the presence of $h(t)=t^{2k}$, we shall denote the corresponding polarization quantities $m^{h}(C)$, $m^{h}_N$,  $M^{h}(C)$, and $M^{h}_N$ from \eqref{max-min} and \eqref{min-max} with $m^{2k}(C)$, $m^{2k}_N$,  $M^{2k}(C)$, and $M^{2k}_N$, respectively, in order to stress that we are in this special case.

Recall the definition \eqref{cell} of the constants $c_{\ell}$, $\ell \geq 0$, and note that $c_{2\ell}$ is equal to the zeroth coefficient (i.e., the 
coefficient $f_0$) in the Gegenbauer expansion of the polynomial $t^{2\ell}$.

\begin {problem}\label {P1}
For given $n \geq 2$ and $k,N\geq 1$,  find quantities  $m^{2k}_N$, $M^{2k}_N$, and the optimal configurations $C$ such that $m^{2k}(C)=m^{2k}_N$ or $M^{2k}(C)=M^{2k}_N$, respectively.
\end {problem}

Let $\Pi_{k,N}^n$, where $n \geq 2$ and $k,N\geq 1$, be the set of all $N$-point spherical codes $C\subset \mathbb{S}^{n-1}$ such that the potential
$$
U_{2k}(x,C):=\sum\limits_{y\in C}(x\cdot y)^{2k}
$$
is constant over $x\in \mathbb{S}^{n-1}$.
For the proof of the results of this section, we will need the following two auxiliary statements. 

\begin {lemma}\label {av}
Let $n \geq 2$ and $k,N\geq 1$. For any $N$-point spherical code 
$C\subset \mathbb{S}^{n-1}$, 
$$
\int_{\mathbb{S}^{n-1}}U_{2k}(x,C)\ \! d\sigma_{n}(x)=c_{2k}N.
$$
\end {lemma}

\begin {proof}
By the Funk-Hecke formula (see \eqref{FunkHecke}), we have
\begin{eqnarray*}
\int_{\mathbb{S}^{n-1}}U_{2k}(x,C)\ \! d\sigma_{n}(x)&=& 
\int_{\mathbb{S}^{n-1}}
\sum\limits_{y\in C}(x\cdot y)^{2k}\ \! d\sigma_{n}(x)\\
&=& \sum\limits_{y\in C}\int_{\mathbb{S}^{n-1}}(x\cdot y)^{2k}\ \! d\sigma_{n}(x)=\sum\limits_{y\in C}\int_{-1}^{1}t^{2k} \, d\mu_{n}(t)=c_{2k}N,
\end{eqnarray*}
which completes the proof.
\end {proof}

\begin {proposition}{\rm(}\cite {Ven2001,KP10}{\rm)}\label {Pi=kk}
Let $n\geq 2$ and $k,N\geq 1$ be such that $\Pi_{k,N}^n\neq \emptyset$. An $N$-point code $C\subset \mathbb{S}^{n-1}$ is a $(k,k)$-design if and only if $C\in \Pi_{k,N}^n$.
\end {proposition}

\begin{proof}
If an $N$-point code $C\subset \mathbb{S}^{n-1}$ is a $(k,k)$-design, by Lemma \ref{k-k-lemma} iii) we have $U_{2k}(x,C)=c_{2k}
N$, $x\in \mathbb S^{n-1}$. Then, by definition, $C\in \Pi_{k,N}^n$. Conversely, if $C\in \Pi_{k,N}^n$, we have $U_{2k}(x,C)=c$, $x\in \mathbb S^{n-1}$, for some constant $c$. Since $\sigma_n$ is a probability measure, by Lemma \ref{av}, we have $c=c_{2k}N$; that is, $\sum\limits_{y\in C}(x\cdot y)^{2k}=c_{2k}N$ for $x\in \RR^n$, $\|x\|=1$. Then by Lemma \ref{k-k-lemma} iii) the code $C$ is a $(k,k)$-design.
\end{proof}

\begin {remark}\label {card}
The set $\Pi_{k,N}^n$ may be empty for some triplets $(n,k,N)$. The result by Seymour and Zaslavsky \cite {SZ} implies that for any $n$ and any $k$, there exists a spherical design of strength $2k$ on $\mathbb{S}^{n-1}$ 
(thus, a spherical $(k,k)$-design) for all but at most finitely many cardinalities $N$. Such designs belong to $\Pi^n_{k,N}$. Therefore, for every pair $(n,k)$, we have $\Pi^n_{k,N}\neq \emptyset$ for all but at most finitely many $N$.
\end {remark}

The following result is a generalization of \cite [Theorem 2]{AmbNie2019}.

\begin{theorem}\label{maxmin}
Let $n \geq 2$ and $k,N\geq 1$. Then for every $N$-point code $C\subset \mathbb{S}^{n-1}$,
\begin {equation}\label {m2}
M^{2k}(C)\geq c_{2k}N.
\end {equation}
Equality in \eqref {m2} holds if and only if $C\in \Pi_{k,N}^n$; that is, if and only if $C$ is an $N$-point $(k,k)$-design on $\mathbb{S}^{n-1}$. If $\Pi_{k,N}^n\neq \emptyset$, we have
\begin {equation}\label {n2}
M^{2k}_{N}=c_{2k}N.
\end {equation}
\end{theorem}

The case $k=1$ of the following theorem is dual to 
\cite[Theorem 2]{AmbNie2019}.

\begin{theorem}\label{minmax}
Let $n \geq 2$ and $k,N\geq 1$. Then for every $N$-point code $C\subset \mathbb{S}^{n-1}$,
\begin{equation}\label{m1}
m^{2k}(C)\leq c_{2k}N .
\end{equation}
Equality in \eqref{m1} holds if and only if $C\in \Pi_{k,N}^n$; that is, if and only if $C$ is a $N$-point $(k,k)$-design on $\mathbb{S}^{n-1}$. If $\Pi_{k,N}^{n}\neq \emptyset$, we have
\begin{equation}\label{n1}
m^{2k}_{N}=c_{2k}N.
\end{equation}
\end{theorem}

Remark \ref {card} implies that for each pair $n \geq 2$, $k\geq 1$, equalities \eqref{n2} and \eqref{n1} hold for all but at most finitely many $N$.

\begin {proof}[Proof of Theorem \ref {maxmin}]
For an arbitrary $N$-point code $C$ on $\mathbb{S}^{n-1}$, taking into account Lemma \ref {av}, we have
$$
M^{2k}(C)=\max\limits_{x\in \mathbb{S}^{n-1}} U_{2k}(x,C)\geq 
\int_{\mathbb{S}^{n-1}}U_{2k}(x,C)\ \! d\sigma_{n}(x)=c_{2k}N;
$$
that is, \eqref {m2} holds. If equality holds in \eqref {m2}, then the maximum value of $U_{2k}(x,C)$ over $\mathbb{S}^{n-1}$ equals its average value over $\mathbb{S}^{n-1}$. If the minimum of $U_{2k}(x,C)$ over $\mathbb{S}^{n-1}$ was strictly less than the maximum, then $U_{2k}(x,C)$ would be strictly less than its maximum on a subset of $\mathbb{S}^{n-1}$ of a positive 
$\sigma_{n}$-measure. Then the average value of $U_{2k}(x,C)$ would be strictly less than its maximum. This contradiction shows that the maximum and minimum value of $U_{2k}(x,C)$ coincide; that is, $U_{2k}(x,C)$ is constant over $\mathbb{S}^{n-1}$. Then $C\in \Pi_{k,N}^n$. 

Conversely, if $C\in \Pi_{k,N}^n$, then $U_{2k}(x,C)$ equals a constant on $\mathbb{S}^{n-1}$. Lemma \ref {av} implies that this constant is $c_{2k}N$. Then $M^{2k}(C)=c_{2k}N$; i.e., equality holds in \eqref {m2}. If $\Pi_{k,N}^n\neq \emptyset$, then equality \eqref {n2} is true.
By Proposition \ref {Pi=kk}, the code $C$ is an $N$-point $(k,k)$-design on $\mathbb{S}^{n-1}$  if and only if $C$ belongs to $\Pi_{k,N}^n$, which occurs if and only if $C$ attains equality in \eqref {m2}.
\end {proof}

The proof of Theorem \ref {minmax} is conducted in a similar way.

\begin {proof}[Proof of Theorem \ref {minmax}]
For an arbitrary $N$-point code $C$ on $\mathbb{S}^{n-1}$, taking into account Lemma \ref {av}, we have
$$
m^{2k}(C)=\min\limits_{x\in \mathbb{S}^{n-1}} U_{2k}(x,C)\leq 
\int_{\mathbb{S}^{n-1}}U_{2k}(x,C)\ \! d\sigma_{n}(x)=c_{2k}N;
$$
that is, \eqref {m1} holds. If equality holds in \eqref {m1}, then the minimum value of $U_{2k}(x,C)$ over $\mathbb{S}^{n-1}$ equals its average value over $\mathbb{S}^{n-1}$. If the maximum of $U_{2k}(x,C)$ over $\mathbb{S}^{n-1}$ was strictly greater than the minimum, then $U_{2k}(x,C)$ would be strictly greater than its minimum on a subset of $\mathbb{S}^{n-1}$ of a positive $\sigma_{n}$-measure. Then the average value of $U_{2k}(x,C)$ would be strictly greater than its minimum. This contradiction shows that the minimum and maximum value of $U_{2k}(x,C)$ coincide; that is, $U_{2k}(x,C)$ is constant over $\mathbb{S}^{n-1}$. Then $C\in \Pi_{k,N}^n$. 

Conversely, if $C\in \Pi_{k,N}^n$, then $U_{2k}(x,C)$ equals a constant on $\mathbb{S}^{n-1}$. Lemma \ref {av} implies that this constant is $c_{2k}N$. Then $m^{2k}(C)=c_{2k}N$; i.e., equality holds in \eqref {m1}. If $\Pi_{k,N}^n\neq \emptyset$, then equality \eqref {n1} holds.
By Proposition \ref {Pi=kk}, the code $C$ is a $N$-point $(k,k)$-design on $\mathbb{S}^{n-1}$ if and only if $C$ belongs to $\Pi_{k,N}^n$, which, in turn, occurs if and only if $C$ attains equality in \eqref {m1}.
\end {proof}

\begin{proof}[Proof of Theorem \ref{minmax-maxmin-extract}]
Theorem \ref{minmax-maxmin-extract} is obtained by combining Theorems \ref{maxmin} and \ref{minmax}.
\end{proof}

\noindent{\bf Acknowledgments.}  
The authors thank an anonymous reviewer for detailed remarks which improved the exposition.

{\bf Funding.}
The research of the second author was supported by Bulgarian NSF grant KP-06-N72/6-2023. The research of the third author was supported, in part, by the Lilly Endowment.
The research of the sixth author was supported, in part, by the European Union-NextGenerationEU, through the National Recovery and Resilience Plan of the Republic of Bulgaria, Grant no. BG-RRP-2.004-0008.

\end{document}